\documentclass[a4paper]{amsart}
\usepackage[utf8]{inputenc}
\usepackage[a4paper, total={6in,8in}]{geometry}

\usepackage{graphicx}
\usepackage{amsmath}
\usepackage{amsthm}
\usepackage{amsfonts}
\usepackage{amssymb}
\usepackage[alphabetic]{amsrefs}
\usepackage{enumerate}
\usepackage{tikz-cd}
\usepackage{float}
\usepackage{mathrsfs}
\usepackage{url}
\usepackage[hidelinks]{hyperref}

\calclayout

\newsavebox{\pullback}
\sbox\pullback{%
\begin{tikzpicture}%
\draw (0,0) -- (1ex,0ex);%
\draw (1ex,0ex) -- (1ex,1ex);%
\end{tikzpicture}}

\theoremstyle{definition}
\newtheorem{rmk}[subsubsection]{Remark}
\newtheorem{definition}[subsubsection]{Definition}
\newtheorem{constr}[subsubsection]{Construction}

\theoremstyle{plain}
\newtheorem{thm}[subsubsection]{Theorem}
\newtheorem*{thm*}{Theorem}
\newtheorem{prop}[subsubsection]{Proposition}
\newtheorem{lemma}[subsubsection]{Lemma}
\newtheorem{cor}[subsubsection]{Corollary}

\DeclareMathOperator\supp{supp}
\DeclareMathOperator\ASpec{Spa}

\DeclareMathOperator\VSpec{Spv}
\DeclareMathOperator\RZop{RZ}
\DeclareMathOperator{\spmap}{sp}
\DeclareMathOperator\FSpec{Spf}
\DeclareMathOperator\AGSpec{Spec}

\DeclareMathOperator{\pr}{pr}

\DeclareMathOperator{\red}{red}

\DeclareMathOperator{\rs}{rs}
\DeclareMathOperator{\fracfield}{Frac}

\newcommand{\Spa}[2]{\ASpec( #1 , #2 )}
\newcommand{\Spv}[1]{\VSpec(#1)}

\newcommand{\Spf}[1]{\FSpec(#1)}
\newcommand{\Spec}[1]{\AGSpec(#1)}
\newcommand{\RZ}[2]{\RZop( #1 , #2 )}
\newcommand{\inv}{^{-1}}
\newcommand{\Oh}{\mathcal{O}}

\newcommand{\tate}[1]{\langle #1 \rangle}
\newcommand{\urs}{_{\rs}}

\newcommand{\ttilde}[1]{\widetilde{#1}}
\newcommand{\bbar}[1]{\overline{#1}}
\newcommand{\what}[1]{\widehat{#1}}
\newcommand{\XX}{\mathfrak{X}}

\newcommand{\ZZ}{\mathfrak{Z}}

\newcommand{\VV}{\mathfrak{V}}
\newcommand{\sA}{\mathscr{A}}

\newcommand{\II}{\mathcal{I}}

\newcommand{\bXX}{\bbar{\XX}}

\title
[Compactifications of rigid analytic spaces through formal models]
{Compactifications of rigid analytic spaces \\ through formal models}
\author{Mateusz Kobak}
\address{Instytut Matematyczny PAN, ul. Śniadeckich 8, 00-656 Warszawa}
\email{mateusz.m.kobak@gmail.com}
\date{\today}

\begin{document}
\begin{abstract}
We provide a~new construction of Huber's universal compactification in the case of the structure morphism of a~quasi-compact, separated rigid analytic space over a non-archimedean field. 
We make use of Raynaud's theory of formal models and Temkin's relative Riemann--Zariski spaces.
\end{abstract}

\maketitle
\tableofcontents

\section{Introduction}

In his book on the \'etale cohomology of adic spaces \cite{HuEt}, Huber constructed a universal way of compactifying a morphism of analytic adic spaces (under some natural assumptions, see \S\ref{ss:huber-comp}). He then successfully used these objects to define cohomology with compact support. Since then, Huber's construction has been widely used in the study of rigid analytic geometry.


However, Huber's construction of compactifications has a reputation for being involved and hard to understand. In this article, we would like to show that it can be seen from a new perspective. Namely, we carry out the construction of these compactifications in terms of Raynaud's formal models. 

Let us recall the following well-known theorem:
\begin{thm}[{\textit{folklore}}]\label{thm:adic_as_limit}
Let $X$ be a quasi-paracompact, quasi-separated rigid space\footnote{To be precise, here and below by a rigid space we mean an adic space locally of finite type over $\Spa{K}{\Oh_K}$ for some non-archimedean field $K$.} and let $\mathcal{M}(X)$ be the category of admissible formal models of $X$ in the sense of Raynaud. Then there is a natural isomorphism of locally ringed spaces
$$
(X, \Oh_X^+) \simeq \varprojlim_{\mathfrak{X} \in \mathcal{M}(X)} (\mathfrak{X}, \Oh_\mathfrak{X}).
$$
\end{thm}
This fact serves as a bridge between two approaches to rigid geometry: between the world of valuations (Huber's adic spaces) and formal algebraic geometry (Raynaud's formal models). Although Huber's compactifications are no longer rigid spaces, but rather more general adic spaces, a sort of analogous bridge can be established for them and we provide it in this article.

Fix a quasi-compact, separated rigid space $X$ and let $\mathcal{M}(\bbar{X})$ be the category of pairs $(\mathfrak{X}, \bbar{\mathfrak{X}}_s)$ such that $\mathfrak{X}$ is an admissible formal model of $X$ and $\bbar{\mathfrak{X}}_s$ is a compactification of the special fiber $\mathfrak{X}_s$ of $\mathfrak{X}$. The main result of this article is the following theorem.
\begin{thm*}[see Theorem {\ref{thm:main}}]
Let $\bbar{X}$ be the universal compactification of $X$ in the sense of \cite{HuEt}. Then there exists a natural isomorphism of locally ringed spaces
$$
(\bbar{X}, \Oh_{\bbar{X}}^+) \simeq \varprojlim_{(\mathfrak{X}, \bbar{\mathfrak{X}}_s) \in \mathcal{M}(\bbar{X})} (\bbar{\mathfrak{X}}_s, \Oh_{(\mathfrak{X}, \bbar{\mathfrak{X}}_s)}),
$$
where $\Oh_{(\mathfrak{X}, \bbar{\mathfrak{X}}_s)}$ is the subsheaf of $\Oh_{\mathfrak{X}}$ consisting of functions that can be extended to $\bbar{\mathfrak{X}}_s$ (see Definition~\ref{def:main} for a precise construction of $\Oh_{(\mathfrak{X}, \bbar{\mathfrak{X}}_s)}$).
\end{thm*}

We remark that in Huber's construction of compactifications, one adds only valuations of rank higher than $1$ (see Theorem \ref{thm:comp}).
In fact, every point $x$ of $\bbar{X}$ lies in the closure of a point $x_0$ of $X$, and hence corresponds to a valuation subring of the residue field $\kappa(x_0) = \Oh_{X, x_0}^+ / \mathfrak{m}_{x_0}^+$. Such subrings naturally form a space called the \textit{Riemann--Zariski space} of $\kappa(x_0)$.
This description is the leading idea of our main proof. 

Our considerations owe a huge intellectual debt to the ideas and the results of Temkin's paper on relative Riemann--Zariski spaces \cite{RZ}, which are summarized briefly in \S\ref{section:rz}.

A result similar to ours, but more difficult to apply in practice, can be obtained using the framework of rigid geometry developed by Fujiwara and Kato \cite{FK}, thanks to the result of Ayoub--Gallauer--Vezzani \cite{AGV}. We explain the differences between the two approaches in \S\ref{section:agv}. We also thank Alberto Vezzani for bringing this result to our attention.

We also would like to mention the work of Gaisin and Welliaveetil \cite{Gaisin} about compactifications of specialization morphisms. We discuss the relation between their theory and our article in \S \ref{section:gaisin}. We also thank Ildar Gaisin for a valuable discussion on this topic.

\subsection{Notation} \label{notation}
For a non-archimedean field $K$ (that is, a field complete with respect to a valuation of rank one) we write $\Oh_K$, $\mathfrak{m}_K$ and $\kappa$ for its ring of integers, maximal ideal of $\Oh_K$ and its residue field, respectively. If $X$ is a rigid space we write $\mathcal{M}(X)$ for the category of its admissible (i.e. flat over $\Spf{\Oh_K}$) formal models.

Let $X$ be an adic space and $x \in X$ a point, then: 
$\mathfrak{m}_x \subset \Oh_{X, x}$ and $\mathfrak{m}^+_x \subset \Oh^+_{X, x}$ are the maximal ideals in respective local rings;
$(k(x), k(x)^+) := (\Oh_{X, x} / \mathfrak{m}_x, \Oh^+_{X, x} / \mathfrak{m}_x)$ is the pair consisting of the residue field at $x$ together with a valuation ring $k(x)^+$;
$\kappa(x) = k(x)^+ / \mathfrak{m}_{k(x)^+} \simeq \Oh_{X, x}^+ / \mathfrak{m}^+_x$ is the residue field of the residue field at $x$;
if $f \in \Oh_{X, x}$ then $|f(x)|$ denotes the value of $f$ with respect to the valuation associated to $x$.  If $X=\Spa{A}{A^+}$ we often write $v_x$ for the associated valuation $v_x: A \to \Gamma_x \cup \{ 0 \}$. Here $\Gamma_x$ is called the \textit{value group} of $v_x$.

If $X$ is a scheme then its residue field at point $x \in X$ is denoted by $\kappa(x)$. For a separated scheme $X$ of finite type over some field $\kappa$ we call a $\kappa$-\textit{compactification} any schematically dense open immersion $X \to \bbar{X}$ of schemes over $\kappa$ with $\bbar{X}$ proper over $\kappa$. 

We often write qcqs for quasi-compact quasi-separated; tft for topologically finite type.

\subsection{Acknowledgements}
The author wants to express his gratitude to Piotr Achinger for all the guidance and help. The author also thanks Marcin Lara for valuable discussions as well as Katharina H\"ubner and Adrian Langer for reading and commenting on the earlier version of this manuscript. This work is a part of the project KAPIBARA supported by the funding from the European Research Council (ERC) under the European Union's Horizon 2020 research and innovation programme (grant agreement No 802787). Most of the results in this paper were contained in the author's MSc project.

\section{Preliminaries}

In this section, we provide the reader with a recap of the necessary tools that we use in our proofs and constructions. There is neither anything new nor illuminating here. Feel free to skip it upon the first read.

\subsection{Rigid geometry and the specialization map} 
Let $K$ be a nonarchimedean field. In this article we consider rigid spaces as analytic adic spaces locally of finite type over $\Spa{K}{\Oh_K}$ (see \cite{HuEt} for a reference on adic spaces) and their formal models as introduced by Raynaud (see \cite{Bosch} or \cite{FK}), which are formal schemes locally of topologically finite type over $\Spf{\Oh_K}$. As we are interested in the interplay of these two theories in terms of compactifications, let us recall a few notions and facts about this phenomenon.

We shall now recall a few definitions that will play an important role in what will follow. For other definitions, we refer the reader to the book of Huber \cite{HuEt}.

\begin{definition}
Let $f: X \to Y$ be a morphism of analytic adic spaces. Then $f$ is called \textit{locally of $^+$weakly finite type} if for every $x \in X$ there exists open affinoid subspaces $U=\Spa{A}{A^+}$, $V=\Spa{B}{B^+}$ of $X$ and $Y$ and a finite subset $E \subset A$, such that: $x \in U$, $f(U) \subset V$, the induced homomorphism $f^{\#} : (B, B^+) \to (A, A^+)$ is of topologically finite type and $A^+$ is the integral closure of $f^{\#}(B)[E \cup A^{\circ \circ}]$ in $A$.
Moreover, $f$ is called \textit{of $^+$weakly finite type} if it is in addition quasi-compact.
\end{definition}

\begin{definition}
Let $f: X \to Y$ be a morphism of adic spaces.
\begin{enumerate}[(a)]
    \item $f$ is called \textit{specializing} if for each point $x$ and each specialization $y'$ of $f(x)$ there exists a specialization $x'$ of $x$ such that $y' = f(x')$.
    \item $f$ is called \textit{universally specializing} if $f$ is locally of weakly finite type and for every adic morphism of adic spaces $Y' \to Y$ the base change $X \times_{Y} Y' \to Y'$ is specializing.
    \item $f$ is called \textit{partially proper} if it is universally specializing, separated and locally of $^+$weakly finite type,
    \item $f$ is called \textit{proper} if it is partially proper and quasi-compact.
\end{enumerate}
\end{definition}

\begin{prop}[{\cite[Proposition 1.3.10]{HuEt}}] \label{prop:adic_vcp}
Let $f: X \to Y$ be a morphism of analytic adic spaces. Then $f$ is partially proper if and only if for any non-archimedean field $L$ and any valuation ring $R \subset \Oh_L \subset L$ when given the solid commutative square
\begin{center}
    \begin{tikzcd}
    \Spa{L}{\Oh_L} \arrow[]{r}{} \arrow[]{d}{} & X \arrow[]{d}{}\\
    \Spa{L}{R} \arrow[]{r}{} \arrow[dashed]{ur}{\exists!}& Y
    \end{tikzcd}
\end{center}
then there exists exactly one dashed arrow.
\end{prop}

Let us now recall the construction of the specialization map.
\begin{constr}[Specialization map] \label{const:spec}
Let $X$ be a rigid space and let $\mathfrak{X} \in \mathcal{M}(X)$ be a formal model of $X$. 
We shall now construct the specialization map 
\[
    \spmap: X \to \mathfrak{X}. 
\]
For this, we only deal with the case of affine $\mathfrak{X}$. The general situation can be reduced to the affine one by considering a cover of $\mathfrak{X}$ by affine opens. 

Let $\mathfrak{X} = \Spf{\mathscr{A}}$ and $X = \Spa{A}{A^\circ}$ for $A \simeq \mathscr{A} \otimes_{\Oh_K} K$. Pick any $x \in X$. Let $k(x)$ be the residue field of $\Oh_{X, x}$ and $k(x)^+ \subset k(x)$ the image of $\Oh_{X, x}^+$ under this quotient. The general fact for adic spaces is that $k(x)^+$ is a valuation ring with fraction field $k(x)$ and moreover $x$ is the image of a unique closed point under a naturally defined morphism 
\[
    \theta_x : \Spa{k(x)}{k(x)^+} \to \Spa{A}{A^\circ} = X.
\]
Let $(\what{k(x)}, \what{k(x)^+})$ be the completion of $(k(x), k(x)^+)$. Note that the completion of valuation rings does not change the residue field, so $\kappa(x)$ is the residue field of $\what{k(x)^+}$ as well. This induces an isomorphism 
\[
    \Spa{\what{k(x)}}{\what{k(x)^+}} \simeq \Spa{k(x)}{k(x)^+}.
\]
The morphism $\theta_x$ induces a natural ring homomorphism $h: A^\circ \to \what{k(x)^+}$ and hence $h|_{\mathscr{A}}: \mathscr{A} \to \what{k(x)^+}$ (as $\mathscr{A}$ is always a subring of $A^\circ$). Now, we define $\spmap(x)$ to be the image of the unique closed point of $\Spf{\what{k(x)^+}} \to \Spf{\mathscr{A}} = \mathfrak{X}$. 
\end{constr}

\begin{rmk}\label{rmk:iota}
    The construction above gives us a natural homomorphism $\iota: \kappa(\spmap(x)) \to \kappa(x)$, where $\kappa(\spmap(x))$ is the residue field of the scheme $\mathfrak{X}_s$ at $\spmap(x)$. 
\end{rmk}

\subsection{Huber's compactification of an adic space} \label{ss:huber-comp}  Here we will recall the properties of Huber's universal compactifications.
\begin{definition}
Assume we have a diagram of analytic adic spaces as below
\begin{center}
    \begin{tikzcd}
        X \arrow[hook]{r}{j} \arrow[swap]{d}{f}
        & \bbar{X} \arrow[]{dl}{f'} \\
        S
    \end{tikzcd}
\end{center}
with $f'$ being a partially proper morphism. Then $j: X \to \bbar{X}$ is called a \textit{universal compactification} of $X$ over $S$ if for each $g: Y \to S$ partially proper and each morphism of adic spaces $h : X \to Y$ over $S$ there exists a unique $h' : \bbar{X} \to Y$ such that the diagram below commutes
\begin{center}
    \begin{tikzcd}
        X \arrow[hook]{dr}{j} \arrow[swap]{dd}{f} 
        \arrow[bend left=30]{drr}{h} \\
        & \bbar{X} \arrow[]{dl}{f'} \arrow[dashed]{r}{h'}
        & Y \arrow[bend left=30]{dll}{g}\\
        S
    \end{tikzcd}
\end{center}
\end{definition}

\begin{thm}[{\cite[Theorem 5.1.5]{HuEt}}] \label{thm:comp}
Let $f: X \to S$ be a morphism of analytic adic spaces, which is separated, locally of $^+$weakly finite type and taut\footnote{Topological space $X$ is called \textit{taut} if the closure of any quasi-compact subset is quasi-compact too.}. Then there exists a universal compactification $\bbar{X}$ together with morphisms $f'$ and $j$ as above. Moreover $\bbar{X}$ together with the morphism $j: X \to \bbar{X}$ is characterized as the unique pair such that the following conditions hold:
\begin{enumerate}[(a)]
    \item $f'$ is partially proper,
    \item $j$ is a quasi-compact open embedding,
    \item each point of $\bbar{X}$ is a specialization of some point in $X$,
    \item the induced map $\Oh_{\bbar{X}} \to j_{*}\Oh_{X}$ is an isomorphism.
\end{enumerate}
\end{thm}

\begin{rmk}
In the Theorem above it is important that actually, conditions \textit{(a)-(d)} imply that $\bbar{X}$ is a universal compactification of $X$. This is how the original proof of Huber is structured -- one first proves that it suffices to construct $\bbar{X}$ such that \textit{(a)-(d)} hold and then one explicitly builds such an object.

In our proof, we will also use the fact that $\bbar{X}$ is characterized by \textit{(a)-(d)} and we will not use the universal property of $\bbar{X}$.
\end{rmk}

The simplest case, where one can explicitly describe Huber compactification is the case of a morphism of affinoids.
 
\begin{prop}[Compactifications of affinoids]\label{prop:compact_affinoid}
Let $A$ be an affinoid $K$-algebra and $X = \Spa{A}{A^\circ}$ the corresponding affinoid rigid space. Let $A^{\min}$ be the minimal integrally closed subring of $A$ that contains $K$ and all the topologically nilpotent elements $A^{\circ\circ} \subset A$. Then $\bbar{X} = \Spa{A}{A^{\min}}$ is the universal compactification of $X$.
\end{prop}

\subsection{Spectral spaces}
We shall now recall the properties of the topology of spectral spaces.
\begin{definition}
A topological space $X$ is called \textit{spectral} if it is quasi-compact and quasi-separated, each irreducible closed subset has a generic point and $X$ has a basis consisting of quasi-compact open subsets.
\end{definition}

\begin{prop}[{\cite[Corollary III.2.4]{Morel}}]
Quasi-compact quasi-separated adic spaces are spectral.
\end{prop}

\begin{definition}
Let $X$ be a spectral space. Then $U \subset X$ is called \textit{constructible} if it is a finite Boolean combination of quasi-compact open subsets and their complements. The \textit{constructible topology} is the topology generated by constructible sets. A subset $Z \subset X$ is called \textit{pro-constructible} if it is closed in the constructible topology. A continuous map $f: X \to Y$ of spectral spaces is called \textit{spectral} if it is quasi-compact (i.e. preimage of quasi-compact open is quasi-compact).
\end{definition}

\begin{prop}[{\cite[Tag 08YF]{Stacks}}] \label{prop:spectral}
Let $f: X \to Y$ be a continuous map of spectral spaces. Then:
\begin{enumerate}[(a)]
\item $X$ with constructible topology is compact Hausdorff,
\item $f$ is spectral if and only if $f$ is continuous in the constructible topology,
\item if $f$ is spectral then $f$ takes pro-constructible sets to pro-constructible sets.
\item if $Z \subset X$ is a pro-constructible subset, then its closure $\bbar{Z}$ consists of all specializations of points in $Z$.
\end{enumerate}
\end{prop}

\begin{prop} \label{p:spectral_limit}
Let $(X_i)_{i \in \mathcal{I}}$ be a cofiltering system of spectral spaces with spectral transition maps. Then:
\begin{enumerate}[(a)]
\item \(X = \lim\limits_{\substack{\longleftarrow \\ i \in \mathcal{I}}} X_i \) is spectral,
\item for any constructible $U \subset X$ there exists $j \in \mathcal{I}$, such that $U = \pr_j\inv(V)$, where $V \subset X_j$ is constructible. If $U$ is open (resp. closed) we can choose $V$ to be open (resp. closed),
\item for any pro-constructible set $Z \subset X$ we have $Z = \bigcap\limits_{i \in \mathcal{I}} \pr_i\inv(\pr(Z))$. 
\end{enumerate}
\end{prop}
\begin{proof}
Claims \textit{(a)} and \textit{(b)} are proved in {\cite[Tag 0A2U]{Stacks}}, but \textit{(c)} is not, so let us prove it. It is obvious that $Z \subset \bigcap\limits_{i \in \mathcal{I}} \pr_i\inv(\pr(Z))$, so we shall prove the other inclusion. Suppose that $x \in RHS$ and $x \not\in Z$. Then there exists a constructible subset $U \subset X$ such that $x \in U$ and $U \cap Z = \emptyset$. Moreover there exists $i_0 \in \mathcal{I}$ together with a constructible subset $V \subset X_{i_0}$ such that $U = \pr_{i_0}\inv (V)$. But since $x \in \pr_{i_0}\inv(\pr_{i_0}(Z))$ then $\pr_{i_0}(x) \in V \cap \pr_{i_0}(Z)$. So there exists $y \in U$ such that $\pr_{i_0}(y) = \pr_{i_0}(z)$ and then $y \in U$. Thus $y \in U \cap Z$, contradiction.
\end{proof}

\subsection{Specialization relations in adic spaces}
In this subsection we recall the theory of vertical and horizontal specializations in the theory of adic spaces. Later on we will work with adic spaces associated to tft $\Oh_K$-algebras, which are not analytic. Thus here we consider the general setting of arbitrary adic spaces.

\subsubsection{Vertical specializations}
Let $A$ be a ring and let $x, y \in \Spv{A}$. Then $y$ is called a \textit{vertical specialization} of $x$ if $y$ is a specialization of $x$ and $\supp(x) = \supp(y)$. Here, for a valuation $v:A \to \Gamma_v \cup \{ 0 \}$, the \textit{support} of $v$ is $\supp(v) = v\inv(0)$.

Morally, vertical specializations are the ones that change the center of a valuation and fix the support. The following two propositions are crucial in understanding their behavior. 

\begin{prop}[{\cite[Proposition I.3.2.3]{Morel}}] \label{p:v_totally_orderd}
Let $x \in \Spv{A}$. Then vertical generalizations of $x$ are totally ordered.
\end{prop}

\begin{prop}\label{prop:vert_spec}
Let $(A, A^+)$ be a Huber pair and let $x \in \Spa{A}{A^+}$. Then there is a canonical order-preserving bijection from the set of vertical specializations of $x$ to $\Spa{\kappa(x)}{\kappa(x)^+}$, where $\kappa(x)^+$ is the image of $A^+$ in $\kappa(x)$. This bijection is constructed as follows.

Let $y$ be a vertical specialization of $x$. Then $y$ corresponds to a valuation $R_y \subset \kappa(x)$ such that for a natural map $i: k(y) \to k(x)$ and a projection $\pi : k(x)^+ \to \kappa(x)$ we have that $k(y)^+ = i\inv (\pi\inv(R_y))$. 
\end{prop}
\begin{proof}
This is a corollary of \cite[Proposition I.3.2.3]{Morel}.
\end{proof}

On rigid spaces (and more generally, on analytic adic spaces), there are only vertical specializations. However, in our constructions in the proof of Theorem \ref{thm:main} we embed a rigid space in a larger adic space, which will indeed have other types of specializations. We shall now define another important class of specializations.

\subsubsection{Horizontal specializations}
Let $x \in \Spv{A}$. The \textit{characteristic group} of $x$ is the convex subgroup of the value group $\Gamma_x$ defined as follows
$$
c\Gamma_x := \left\langle \{\gamma \in \Gamma_x : \gamma \geq 1 \text{ and } \exists_{a \in A} |a(x)| = \gamma \}\right\rangle \subset \Gamma_x,
$$
where $\langle \cdot \rangle$ means that we generate a subgroup by these elements.

Let $x \in \Spv{A}$ and $H \subset \Gamma_x$ be a convex subgroup such that $c\Gamma_x \subset H$, then we define $x|_H$ as follows:
$$
|f(x|_H)| = 
\begin{cases} 
\begin{array}{cl}
|f(x)| & \text{ if } |f(x)| \in H, \\
0 & \text{ otherwise}.
\end{array}
\end{cases}
$$
Now if $y \in \Spv{X}$ is a specialization of $x$ then $y$ is called a \textit{horizontal specialization} of $x$ if it is of the form $x|_H$ as above.

In our proof most of the technical work that deal with specializations is done in Lemma \ref{lemma:sp_rz_properties}. There we make use of the following facts.

\begin{prop} \label{prop:horizontal}
Let $x \in \Spa{A}{A^+} \subset \Spv{A}$. Then for any horizontal specialization $y \in \Spv{A}$ of $x$ we have $y \in \Spa{A}{A^+}$. Moreover, horizontal specializations of $x$ are totally ordered with $x|_{c\Gamma_x}$ being the minimal horizontal specialization.
\end{prop}
\begin{proof}
This is \cite[Proposition II.2.3.1]{Morel} and \cite[Proposition 4.18]{Wed}.
\end{proof}

\begin{prop}[{\cite[4.19]{Wed}}]\label{spec_factorization}
Let $x,y \in \Spv{A}$ such that $y \preceq x$. Then there exists $x' \in \Spv{A}$ such that $x \succeq x'$ is a vertical specialization and $x' \succeq y$ is a horizontal specialization.
\end{prop}

\begin{prop}[Adic going-up {\cite[Corollary 4.20]{Wed}}] \label{adic_going_up}
Let $y \in \Spv{A}$ and $\mathfrak{p} \subset \supp(y)$ be a prime ideal. Then there exists a horizontal generalization $y \preceq x \in \Spv{A}$ such that $\supp(x) = \mathfrak{p}$. 
\end{prop}

\begin{prop}\label{hor_cont}
Let $A$ be a ring with $I$-adic topology, for some ideal $I \subset A$. Let $x \in \Spv{A}$ and $y$ be its minimal horizontal specialization. If $\supp(y) \in V(I)$ then there exists a horizontal specialization $x' \preceq x$ which is continuous and $\supp(x') \not\in V(I)$. In particular, $x'$ is unique and is minimal among horizontal specializations of $x$ that have support outside $V(I)$. 
\end{prop}
\begin{proof}
This is the discussion of sections 7.1 and 7.2 in \cite{Wed}.
\end{proof}

\subsection{Relative Riemann--Zariski spaces}\label{section:rz}

In \cite{RZ} Temkin introduced the notion of relative Riemann--Zariski spaces. Here we recall the definition and the properties of Riemann--Zariski spaces that will be relevant for our proof.

\begin{definition}
Let $f: Y \to X$ be a separated morphism between quasi-compact quasi-separated schemes. Consider factorization of $f$ as on the diagram below
\begin{center}
	\begin{tikzcd}
	    Y \arrow[]{dd}{f} \arrow[]{dr}{f_i}\\
	    & X_i \arrow[]{dl}{g_i}\\
	    X
	\end{tikzcd}
\end{center}  
    such that $g_i$ is proper and $f_i$ is dominant (i.e. image of $f_i$ is dense). Then $X_i$ is called a $Y$-\textit{modification} of $X$ and the \textit{Riemann--Zariski space} is defined as the limit $\RZ{Y}{X} := \lim\limits_{\longleftarrow} X_i$ in the category of locally ringed spaces, where $X_i$ runs over all $Y$-modifications of $X$. If $A$ and $B$ are rings, we often write $\RZ{B}{A}$ instead of $\RZ{\Spec{B}}{\Spec{A}}$.
\end{definition}

The above system $(X_i)$ is cofiltering, so everything behaves in a nice way. Observe, that here we equipped $\RZ{Y}{X}$ with a structure sheaf $\Oh_{\RZ{Y}{X}} = \lim\limits_{\longrightarrow} \Oh_{X_i}$ (with a little abuse of notation, as we need to take inverse images of the sheaves inside the colimit). Let $\eta: Y \to \RZ{Y}{X}$ be a natural map to the inverse limit. Then we define a \textit{sheaf of meromorphic functions} as $\mathcal{M}_{\RZ{Y}{X}} = \eta_* \Oh_Y$. There is a relation between the aforementioned two sheaves on $\RZ{Y}{X}$, which is specified by the next proposition.

\begin{prop}[{\cite[Corollary 3.5.2]{RZ}}]\label{rz_sheaf}
The map $\eta$ is injective, each point $x \in \RZ{Y}{X}$ has a unique minimal generalization $y \in \eta(Y)$, $\mathcal{M}_{\RZ{Y}{X}, x} \simeq \Oh_{Y, y}$. Moreover there exists a valuation ring $R \subset \Oh_{Y, y} / \mathfrak{m}_y$ whose preimage in $\Oh_{Y, y}$ is exactly $\Oh_{\RZ{Y}{X}, x}$.
\end{prop}

Historically, Riemann--Zariski spaces were first introduced in a special case by Zariski to study resolutions of singularities. In his special case, Zariski was able to identify points on $\RZ{Y}{X}$ with certain isomorphism classes of valuations. Since points on adic spaces also correspond to valuations one should ask what is the relation between these two phenomena. The answer is as follows.

\begin{prop}[{\cite[Corollary 3.4.7]{RZ}}] \label{rz_subspace}
Let $A \subset B$ be rings with the discrete topology. Then the underlying topological space of $\RZ{B}{A}$ is the subspace of the adic space $\Spa{B}{A}$ consisting of points with no horizontal specialization.
\end{prop}

\begin{rmk} \label{rmk:rz}
\textit{(i)} Above we gave an interpretation of the Riemann--Zariski space as certain space of valuations with support on $Y$ and center on $X$. If $x \in \RZ{Y}{X}$ then the corresponding valuation is centered on the minimal generalization $y$ in $\eta(Y)$ with the valuation ring being the image of $\Oh_{\RZ{Y}{X}}$ in the residue field $\Oh_{Y, y} / \mathfrak{m}_y$ (see the proof of \cite[Corollary 3.5.2]{RZ}).

\textit{(ii)} The above statement is still true if we replace rings $A$ and $B$ by schemes $X$ and $Y$. Then $\Spa{Y}{X}$ is an adic space whose points correspond to all valuations with support on $Y$ and center on $X$. For a more detailed discussion on adic spaces associated to morphisms of schemes see \cite{Hubner}.
\end{rmk}

Our main construction will be in the spirit of the next proposition.

\begin{prop}{\cite[Corollary 3.3.6]{RZ}} \label{prop:temkin_retraction}
Let $A \subset B$ be rings with the discrete topology. Then the function $r : \Spa{B}{A} \to \RZ{B}{A}$ taking each $x \in \Spa{B}{A}$ to its minimal horizontal specialization is continuous.
\end{prop}

In particular, if $B \simeq k$ is a field, then we have a homeomorphism $ \RZ{k}{A} \simeq \Spa{k}{A}$. Thanks to Nagata's theorem on the existence of compactifications we conclude what follows.

\begin{cor}\label{rz_comp}
Let $k$ be a field and $X$ a separated scheme of finite type over $k$. Then $\RZ{X}{k} \simeq \lim \bbar{X_i}$, where the limit goes over all $k$-compactifications of $X$. 
\end{cor}

We finish this section with the proposition that captures the interplay of Riemann--Zariski spaces and closures of single points on analytic adic spaces.

\begin{prop}\label{prop:rz_as_closure}
 Let $K$ be a nonarchimedean field. Let $X$ be an analytic adic space of $^+$weakly finite type and proper over $\Spa{K}{\Oh_K}$ and let $x \in X$ be any point. Then the points on the topological closure $\bbar{\{x\}}$ of $x$ are in natural bijection with points on $\RZ{\kappa(x)}{\kappa}$, where $\kappa$ is the residue field of $K$.
\end{prop}
\begin{proof}
This is an easy corollary of \cite[Section 1.3]{HuEt}. The map $\rho: \bbar{\{x\}} \to \RZ{\kappa(x)}{\kappa}$ assigns to a point $y$ the valuation $v_y \in \RZ{\kappa(x)}{\kappa}$ corresponding to the valuation ring $R_y$ given in Proposition \ref{prop:vert_spec}. Bijectivity of this map is a result of the adic valuative criterion for properness (Proposition \ref{prop:adic_vcp}).
\end{proof}

\section{Compactifications through formal models}

In this section, we construct Huber's universal compactifications of rigid spaces in a new way. We use the theories of formal models and relative Riemann--Zariski spaces.

Throughout this section, we fix a non-archimedean field $K$ and by $\Oh_K$ and $\kappa$ we mean its ring of integers and residue field, respectively. The notation that we use is described in \S\ref{notation}.

\subsection{Statement of the main theorem}
\begin{definition} \label{def:main}
A \textit{compactified formal scheme} over $\Oh_K$ is a pair $(\mathfrak{X}, \overline{\mathfrak{X}}_{s})$, where $\mathfrak{X}$ is an admissible formal scheme over $\Oh_K$ and $\overline{\mathfrak{X}}_{s}$ is a schematic $\kappa$-compactification of the special fiber $\mathfrak{X}_s = \mathfrak{X} \otimes_{\Oh_K} \kappa$. Morphisms of compactified formal schemes $f: (\mathfrak{X'}, \overline{\mathfrak{X'}}_{s}) \to (\mathfrak{X}, \overline{\mathfrak{X}}_{s})$ are commutative diagrams of the form:
\begin{center}
    \begin{tikzcd}
    \mathfrak{X}'  \arrow[]{d}{} \arrow[hookleftarrow]{r}{} & \mathfrak{X}'_{s} \arrow[hook]{r}{} \arrow[]{d}{}&\overline{\mathfrak{X}'}_{s}  \arrow[]{d}{} \\
    \mathfrak{X} \arrow[hookleftarrow]{r}{} & \mathfrak{X}_{s} \arrow[hook]{r}{}& \overline{\mathfrak{X}}_{s} 
    \end{tikzcd}
\end{center}
where arrows going to the left are natural closed immersions and arrows going to the right are open embeddings into compactifications.

For a rigid space $X$ we define the category $\mathcal{M}(\overline{X})$ of \textit{compactified admissible formal models} of $X$ to consist of compactified formal schemes $(\mathfrak{X}, \overline{\mathfrak{X}}_{s})$, such that $\mathfrak{X} \in \mathcal{M}(X)$ is an admissible formal model of $X$. The category $\mathcal{M}(\bbar{X})$ is a cofiltering poset (see Lemma \ref{lemma:cofiltering} below), and hence we can set
$$
\ttilde{X} := \lim\limits_{\substack{\longleftarrow \\ (\mathfrak{X}, \overline{\mathfrak{X}}_{s}) \in \mathcal{M}(\overline{X})}} \overline{\mathfrak{X}}_{s},
$$
where the limit is taken in the category of topological spaces. The space $\ttilde{X}$ can be equipped with a sheaf of rings $\Oh_{\ttilde{X}}$ in the following way. Let $(\mathfrak{X}, \overline{\mathfrak{X}}_s) \in \mathcal{M}(\overline{X})$ together with open immersion $\ttilde{j}: \mathfrak{X}_s \to \overline{\mathfrak{X}}_s$. We define $\Oh_{(\mathfrak{X}, \bbar{\mathfrak{X}}_s)}$ to be the pullback in the category of sheaves of rings on $\overline{\mathfrak{X}}_{s}$:
\begin{center}
    \begin{tikzcd}
    \Oh_{(\mathfrak{X}, \bbar{\mathfrak{X}}_s)} \arrow[]{r}{} \arrow[hook]{d}{} \arrow[dr, phantom, "\square"]
    & \Oh_{\overline{\mathfrak{X}}_s} \arrow[hook]{d}{}\\
    \ttilde{j}_* \Oh_{\mathfrak{X}} \arrow[]{r}{}& \ttilde{j}_* \Oh_{\mathfrak{X}_s}
    \end{tikzcd}
\end{center}
It is easy to check that a morphism of compactified formal schemes $f: (\mathfrak{X'}, \overline{\mathfrak{X}'}_{s}) \to (\mathfrak{X}, \overline{\mathfrak{X}}_{s})$ induces a canonical map of sheaves $f\inv \Oh_{(\mathfrak{X}, \bbar{\mathfrak{X}}_s)} \to \Oh_{(\mathfrak{X}', \bbar{\mathfrak{X}}'_{s})}$. Then we define 
\[
    \Oh_{\ttilde{X}} = \lim\limits_{\longrightarrow} \Oh_{(\mathfrak{X}, \bbar{\mathfrak{X}}_s)}
\]
(this is a small abuse of notation, as we should take the limit of preimages of $\Oh_{(\mathfrak{X}, \bbar{\mathfrak{X}}_s)}$ along natural projections). 
\end{definition}

\begin{rmk}\label{rmk:is_formal_scheme}
    The locally ringed space $(\bXX_s, \Oh_{(\mathfrak{X}, \overline{\mathfrak{X}}_{s})})$ is actually a formal scheme. See Appendix \ref{section:app} for a proof of a more general statement.
\end{rmk}

Let $j: X \to \bbar{X}$ be the natural open immersion of $X$ into its universal compactification (in the sense of Theorem \ref{thm:comp}). Our goal is to prove the following theorem.
\begin{thm}\label{thm:main}
Let $X$ be a quasi-compact separated rigid space. Then there is a naturally defined isomorphism of locally ringed spaces
\[
    \varphi: \left(\overline{X}, \Oh_{\overline{X}}^+\right) \stackrel{\sim}\to \left(\ttilde{X}, \Oh_{\ttilde{X}}\right).
\]
\end{thm}
See Construction \ref{const:map} below for a description of $\varphi$ at the level of points. We will first prove that this map is a homeomorphism of topological spaces (Proposition \ref{prop:homeo}). Then, we notice that using this isomorphism, both sheaves $\Oh_{\ttilde{X}}$ and $\Oh_{\bbar{X}}^+$ can be regarded as subsheaves of $j_* \Oh_{X}^+$, so to conclude the proof, we show that they are equal (Proposition \ref{prop:iso_sheaves}).

\subsection{Comparison of topological spaces}
\begin{constr}\label{const:map}
Let $X$ be a quasi-compact separated rigid space and pick $(\mathfrak{X}, \overline{\mathfrak{X}}_{s}) \in \mathcal{M}(\overline{X})$. We are going to construct a map $\bbar{\spmap}: \overline{X} \to \overline{\mathfrak{X}}_{s}$ extending the construction in Construction \ref{const:spec}. Let us take any $x \in \bbar{X}$. Then by Theorem \ref{thm:comp} there exists a vertical generalization of $x$ that is contained in $X$ (recall that all specializations on analytic adic spaces are vertical). Let $x_0$ be minimal (i.e. most special) among those. Observe that by Proposition \ref{prop:vert_spec} $x$ defines a valuation ring on $R_x \subset \kappa(x_0)$. Moreover by Remark \ref{rmk:iota} there exists a natural homomorphism $\iota: \kappa(\spmap(x_0)) \to \kappa(x_0)$. This gives us a commutative diagram
\begin{center}
    \begin{tikzcd}
    \Spec{\kappa(x_0)} \arrow[]{r}{} \arrow[]{d}{} & 
    \overline{\mathfrak{X}_s} \arrow[]{d}{} \\
    \Spec{R_x} \arrow[]{r}{} \arrow[dashed]{ur}{\theta}& 
    \Spec{\kappa}
    \end{tikzcd}
\end{center}
Here, there exists exactly one $\theta$ by the valuative criterion of properness. We set $\bbar{\spmap}(x)$ to be the image of a closed point under $\theta$. Moreover, observe, that if $x \in X$ then $x_0 = x$ and thus we obtain $\spmap(x) = \bbar{\spmap}(x)$.

Since $\ttilde{X} = \lim\limits_{\longleftarrow} \overline{\mathfrak{X}}_{s}$ we obtain the desired map of sets $\varphi: \overline{X} \to \ttilde{X}$.
\end{constr}

In this section, we focus on proving the following statement.

\begin{prop}\label{prop:homeo}
For a quasi-compact, separated rigid space $X$ the map $\varphi: \overline{X} \to \ttilde{X}$ constructed above is a homeomorphism.
\end{prop}

Before we come to the proof we need to understand more about the space $\ttilde{X}$. We will play with spectral spaces quite a lot, so we naturally want to prove the following lemma.

\begin{lemma} \label{lemma:cofiltering}
The category $\mathcal{M}(\overline{X})$ is a cofiltering poset.
\end{lemma}
\begin{proof}
Let $(\mathfrak{X}_1, \overline{\mathfrak{X}}_{1, s}), (\mathfrak{X}_2, \overline{\mathfrak{X}}_{2, s}) \in \mathcal{M}(\bbar{X})$. By \cite[Proposition II.1.3.1]{FK}, the category $\mathcal{M}(X)$ is cofiltering, and hence there exists $\mathfrak{X} \in \mathcal{M}(X)$ together with maps $s_i: \mathfrak{X} \to \mathfrak{X}_i$, where $s_i$ are admissible formal blow-ups. Pick any compactification $j: \mathfrak{X}_s \hookrightarrow \bbar{\mathfrak{X}}_s$. Now $(j, s_1, s_2)$ define a map $h: \mathfrak{X}_s \to \bbar{\mathfrak{X}}_s \times \overline{\mathfrak{X}}_{1, s} \times \overline{\mathfrak{X}}_{2, s}$. Let $Y$ be the schematic image of $h$. Then it is clear that $(\mathfrak{X}, Y) \in \mathcal{M}(\bbar{X})$ dominates $(\mathfrak{X}_1, \overline{\mathfrak{X}}_{i, s})$ for $i=1, 2$.

By the theorem of Nagata on the existence of compactifications the system $\mathcal{M}(\overline{X})$ is nonempty. Moreover, it has no parallel arrows thanks to separatedness and density assumptions.
\end{proof}

\begin{cor}
The space $\ttilde{X}$ is spectral.
\end{cor}
\begin{proof}
The space $\ttilde{X}$ is a cofiltered limit of spectral spaces, where transition maps are spectral, hence it is spectral by \cite[Tag 0A2U]{Stacks}.
\end{proof}

We shall now construct the map $\bbar{\spmap}$ in a different way for affinoid $X$ and affine $\mathfrak{X}$ (Construction~\ref{const:sp_prime}). It will not be clear at first that these constructions coincide. We prove that they are identical in Corollary \ref{cor:equal_maps}. The reason for two different constructions is that the author is able to carry out a global construction only along the lines of Construction \ref{const:map}. On the other hand, the map defined in Construction \ref{const:sp_prime} is more natural and it is easier to verify its properties. It is an interesting question whether the procedure below can be somehow globalized.

\begin{constr} \label{const:phi}
Let $A$ be an affinoid $K$-algebra and let $\overline{X} = \Spa{A}{A^{\text{min}}}$ be the universal compactification of $X = \Spa{A}{A^\circ}$ (see Proposition \ref{prop:compact_affinoid}).
Let $\mathscr{A}$ be some admissible $\Oh_K$-algebra of topologically finite type, such that $\mathscr{A} \otimes_{\Oh_K}K \simeq A$. Let $\mathscr{A}^{\text{min}} = \mathscr{A} \cap A^{\text{min}}$ and $\mathscr{A}_\kappa = \mathscr{A} \otimes_{\Oh_K} \kappa$. We are going to construct a map $\spmap_{\RZop} : \overline{X} \to \RZ{\mathscr{A}_\kappa}{\kappa}$, where $\RZ{\mathscr{A}_\kappa}{\kappa}$ is the Riemann--Zariski space of $\mathscr{A}_\kappa$ over $\kappa$. For simplicity we write $X_{\RZop} := \RZ{\mathscr{A}_\kappa}{\kappa}$, $X^\circ := \Spa{\mathscr{A}}{\mathscr{A}^{\text{min}}}$. We will construct $\spmap_{\RZop}$ as the composition of $\spmap_{\RZop} =r \circ i$ for certain maps $r$ and $i$ as on the diagram below.
\begin{center}
    \begin{tikzcd}
        \overline{X} = \Spa{A}{A^{\text{min}}} \arrow[]{r}{i} \arrow[bend right=15, swap]{rr}{\spmap_{\RZop}}& \Spa{\mathscr{A}}{\mathscr{A}^{\text{min}}} \arrow[]{r}{r}& \text{RZ}(\mathscr{A}_\kappa, \kappa) = X_{\RZop}
    \end{tikzcd}
\end{center}
First, we take $i: \Spa{A}{A^{\text{min}}} \to \Spa{\mathscr{A}}{\mathscr{A}^{\text{min}}} = X^\circ$ to be the obvious open immersion on the rational subset $X^\circ\left( \frac{1}{\varpi}\right)$ for any choice of pseudouniformizer $\varpi \in \mathscr{A}$. This is exactly the analytic locus of $X^\circ$. 
Note that we have immersions of topological spaces:
$$
\RZ{\mathscr{A}_\kappa}{\kappa} \subset \Spa{\mathscr{A}_\kappa}{\kappa} \subset \Spa{\mathscr{A}}{\mathscr{A}^{\text{min}}}, 
$$
where the first one follows from Proposition \ref{rz_subspace} and the second is the closed immersion coming from the quotient map of Huber pairs $(\mathscr{A}, \mathscr{A}^{\text{min}}) \to (\mathscr{A}_\kappa, \kappa)$. 
Now, we want $r: \Spa{\mathscr{A}}{\mathscr{A}^{\text{min}}} \to \RZ{\mathscr{A}_\kappa}{\kappa}$ to be a retraction, which assigns to every valuation its minimal horizontal specialization. From Proposition \ref{prop:temkin_retraction} we know that each point in $\Spa{\mathscr{A}_\kappa}{\kappa}$ has a horizontal specialization in $\RZ{\mathscr{A}_\kappa}{\kappa}$. So, for $r$ to be well-defined as a function, it suffices to prove, that each analytic point $x \in \Spa{\mathscr{A}}{\mathscr{A}^{\text{min}}}$ specializes horizontally to a point in $\Spa{\mathscr{A}_\kappa}{\kappa}$. This is shown in the next lemma.
\end{constr}

\begin{lemma}
For each $x \in \Spa{\mathscr{A}}{\mathscr{A}^{\text{min}}}$ its minimal horizontal specialization lies in $\Spa{\mathscr{A}_\kappa}{\kappa}$.
\end{lemma}
\begin{proof}
Recall that by Proposition \ref{prop:horizontal} the minimal horizontal specialization of $x$ is $x|_{c\Gamma_x}$ (here $c\Gamma_x$ is taken with respect to $\mathscr{A}$). The claim that $x|_{c\Gamma_x} \in \Spa{\mathscr{A}_\kappa}{\kappa}$ is equivalent to saying that there is no pseudouniformizer $\varpi \in \Oh_K$ such that $\varpi \in c\Gamma_x$. To prove it, suppose that $\varpi \in c\Gamma_x$. Then there exists $a \in \mathscr{A}$ such that $|a \cdot \varpi| > 1$. But $\mathscr{A} \subset A^\circ$, so in particular the element $a$ is powerbounded. But this is in contradiction with $|a \cdot \varpi| > 1$. So $\varpi \not\in c\Gamma_x$ and we indeed have that $x|_{c\Gamma_x} \in \Spa{\mathscr{A}_\kappa}{\kappa}$.
\end{proof}

Let us now investigate some properties of the map $\spmap_{\RZop}$.

\begin{lemma} \label{lemma:sp_rz_properties}
The map $\spmap_{\RZop} : \overline{X} \to \RZ{\mathscr{A}_\kappa}{\kappa}$ constructed above is continuous, spectral, closed, and surjective.
\end{lemma}
\begin{proof}
We shall first prove that $\spmap_{\RZop}$ is continuous. It is sufficient to check it for the map $r$. Pick any $\widetilde{f}, \widetilde{g} \in \mathscr{A}_\kappa$ and consider the rational set $X_{\RZop} \left( \frac{\widetilde{f}}{\widetilde{g}}\right)$. Since rational sets like this generate the topology in $X_{\RZop}$, it is enough to prove that $r\inv \left( X_{\RZop} \left( \frac{\widetilde{f}}{\widetilde{g}}\right) \right)$ is open. Take any lifts $f, g \in \mathscr{A}$ of $\widetilde{f}$ and $\widetilde{g}$ respectively. We claim the following
$$
r\inv \left( X_{\RZop} \left( \frac{\widetilde{f}}{\widetilde{g}}\right) \right) 
= 
\bigcup_{h \in \mathscr{A}} \left( X^\circ \left( \frac{f}{g}\right) \cap X^\circ \left( \frac{1}{g \cdot h}\right) \right).
$$

Let $x \in RHS$. There exists $h \in \mathscr{A}$ such that $|gh(x)| \geq 1$. So $|gh(x)| \in c\Gamma_x$ and thus $|gh(r(x))| \geq 1 $, hence $|g(r(x))| \not= 0$. Obviously $|f(r(x))| \leq |g(r(x))|$, by the construction of horizontal specialization. Now, as the valuation $r(x)$ takes $A^{\circ\circ}$ to zero, it is a valuation on $\mathscr{A}_\kappa$ and we have $\widetilde{f}(r(x)) = f(r(x))$ and $\widetilde{g}(r(x)) = g(r(x))$. So $\ttilde{f}(r(x)) \leq \ttilde{g}(r(x))$ and hence $x \in LHS$.

Let $x \in LHS$. Since $|g(r(x))| \not= 0$ and by definition of $c\Gamma_x$ there exists $h \in \mathscr{A}$ such that $|gh(r(x))| \geq 1$, hence $|gh(x)| \geq 1$. Moreover, obviously $|f(r(x))| \leq |g(r(x))|$ and thus, since $g(r(x)) \not=0$ and $c\Gamma_x$ is convex we get $|f(x)| \leq |g(x)|$. Hence $x \in RHS$.

Moreover observe that actually 
\[
    X^\circ \left( \frac{f}{g}\right) \cap X^\circ \left( \frac{1}{g \cdot h}\right) = X^\circ \left( \frac{fh, 1}{gh} \right),
\]
which is a rational quasi-compact open subset, since $\{ fh, gh, 1\}$ generate an open ideal.

We shall now prove the spectrality of $\spmap_{\RZop}$. Since we assumed that $\mathscr{A}$ is topologically of finite type we can pick its generators $h_1, \ldots, h_n$. Assume that for some $x \in X^\circ$ we have $|gh_i(x)| \geq 1$ for all $i = 1,\ldots,n$. Then, since $\Oh_K[h_1, \ldots, h_n]$ is dense in $\mathscr{A}$, we conclude that for all $h \in \mathscr{A}$ we have $|gh(x)| \geq 1$. So we have
$$
\bigcup_{h \in \mathscr{A}} X^\circ \left( \frac{1}{g \cdot h}\right)
=
\bigcup_{i=1}^n X^\circ \left( \frac{1}{g \cdot h_i}\right)
$$
and moreover
$$
r\inv \left( X_{\RZop} \left( \frac{\widetilde{f}}{\widetilde{g}}\right) \right) = 
\bigcup_{i=1}^n \left( X^\circ \left( \frac{f}{g}\right) \cap X^\circ \left( \frac{1}{g \cdot h_i}\right) \right),
$$
hence preimage under $r$ of a quasi-compact open is a quasi-compact open, i.e. $r$ is spectral. Since $i$ is spectral too, we conclude that their composition $\spmap_{\RZop}$ is spectral.

We now come to the surjectivity of $\spmap_{\RZop}$. Let $y \in \Spa{\mathscr{A}}{\mathscr{A}^{\text{min}}}$ and by $\mathfrak{p}_y \subset \mathscr{A}$ denote the kernel of the associated valuation $v_y$ on $\mathscr{A}$. Since $\mathscr{A}$ is flat over $\Oh_K$ the map $\mathscr{A} \to A$ is injective. Thus the associated map on spectra $u: \Spec{A} \to \Spec{\mathscr{A}}$ is a quasi-compact open immersion with dense image. Hence there exists $\mathfrak{p} \in u(\Spec{A})$ such that $\mathfrak{p}_y$ is a specialization of $\mathfrak{p}$. 
Observe that $\mathfrak{p}$ is not an open ideal of $\mathscr{A}$. By the adic going-up (Proposition \ref{adic_going_up}) there exists a horizontal generalization $x \in \Spv{\mathscr{A}}$ of $y$ such that $\mathfrak{p}_x = \mathfrak{p}$. Let $I$ be some ideal of definition of $\mathscr{A}$. Now, take $x'$ to be a horizontal specialization of $x$, minimal among those that have support outside of $V(I) \subset \Spec{\mathscr{A}}$. Note that horizontal specializations are totally ordered, so $x'$ specializes horizontally to $y$. 
By Proposition \ref{hor_cont} $x'$ is continuous. So in order to prove that $x' \in \Spa{\mathscr{A}}{\mathscr{A}^{\text{min}}}$, it is sufficient to show, that for all $a \in \mathscr{A}^{\text{min}}$ we have $|a(x')| \leq 1$. Suppose contrary, and let $|a(x')| > 1$ for some $a \in \mathscr{A}^{\text{min}}$. 
Then, $|a(x')| \in c\Gamma_{x'}$ (the characteristic group is considered here with respect to $\mathscr{A}$) and since $y$ is a horizontal specialization of $x'$ we get $|a(y)| > 1$, which is a contradiction. So $x' \in \Spa{\mathscr{A}}{\mathscr{A}^{\text{min}}}$ and its support does not contain $I$, hence $x'$ is an analytic point, i.e. $x' \in \Spa{A}{A^{\text{min}}}$. When we set $y \in \RZ{\mathscr{A}_\kappa}{\kappa}$, then $y = r(x')$. So $\spmap_{\RZop}$ is indeed surjective.

Let $Z \subset \Spa{A}{A^{\text{min}}}$ be a closed subset and $Z'$ its closure in $\Spa{\mathscr{A}}{\mathscr{A}^{\text{min}}}$. Then since $r$ is a retraction on specializations we have $r(Z') = Z' \cap \RZ{\mathscr{A}_\kappa}{\kappa}$ which is obviously closed in $\RZ{\mathscr{A}_\kappa}{\kappa}$. So it suffices to prove that $\spmap_{\RZop}(Z) = r(Z')$. Note that $\Spa{A}{A^{\text{min}}}$ is quasi-compact, so $Z$ is pro-constructible in $\Spa{\mathscr{A}}{\mathscr{A}^{\text{min}}}$. Thus by Proposition \ref{prop:spectral} each point in $Z'$ is a specialization of a point in $Z$. 

Let $x \in Z$ and suppose $y \in Z'$ is its specialization. Then by Proposition \ref{spec_factorization} there exists $x' \in \Spv{\mathscr{A}}$, such that $x'$ is a vertical specialization of $x$ and $y$ is a horizontal specialization of $x'$. Let $I$ be some ideal of definition of $\mathscr{A}$. Note that the support of $x'$ lies outside $V(I)$. Thus, there exists a minimal horizontal specialization $x''$ of $x'$, whose support lies outside $V(I)$. By identical considerations as previously, one can show that $x'' \in \Spa{A}{A^{\text{min}}}$ and $y$ is a horizontal specialization of $x''$. Hence, $r(x'') = r(y)$ and we conclude that $\spmap_{\RZop}$ is closed.
\end{proof}

\begin{rmk}
The continuity of the map $\spmap_{\RZop}: \Spa{A}{A^{\text{min}}} \to \RZ{\mathscr{A}_\kappa}{\kappa}$ can be also seen from a different perspective. Consider the diagram:
\begin{center}
    \begin{tikzcd}
    \Spa{A}{A^{\text{min}}} \arrow[hook]{r}{} \arrow[bend right=7]{drrr}{} &
    \Spa{\mathscr{A}}{\mathscr{A}^{\text{min}}} \arrow[hook]{r}{disc} &
    \Spa{\mathscr{A}_{disc}}{\mathscr{A}^{\text{min}}_{disc}} \arrow[]{r}{r'} &
    \RZ{\mathscr{A}_{disc}}{\mathscr{A}^{\text{min}}_{disc}} \\
    & & & \RZ{\mathscr{A}_\kappa}{\kappa} \arrow[hook]{u}{}
    \end{tikzcd}
\end{center}
where subscript $disc$ means taking the same ring, but with discrete topology, and $r'$ is the continuous retraction discussed in Proposition \ref{prop:temkin_retraction}. The diagram commutes, since the map $disc$ will not add any new horizontal specializations, because a horizontal specialization of a continuous valuation is always continuous by Proposition \ref{prop:horizontal}.
\end{rmk}

\begin{constr}\label{const:sp_prime}

Observe that the above map $\spmap_{\RZop}: \Spa{A}{A^{\text{min}}} \to \RZ{\mathscr{A}_\kappa}{\kappa}$ can be composed with the natural projection $\pr: \RZ{\mathscr{A}_\kappa}{\kappa} \to {\bbar{\mathfrak{X}}_s}$, where ${\bbar{\mathfrak{X}}_s}$ is some $\kappa$-compactification of $\mathfrak{X}_s = \Spec{\mathscr{A}_\kappa}$. Note that $\pr$ is surjective since any map of two compactifications is birational and proper. Moreover, $\pr_Y$ is spectral and closed by \cite[Theorem II.E.2.5]{FK}. Thus we get a continuous spectral closed surjection $\bbar{\spmap}': \pr \circ \spmap_{\RZop} : \Spa{A}{A^{\text{min}}} \to {\bbar{\mathfrak{X}}_s}$.

Notice, that $\pr_Y$ takes an element $x \in \RZ{\mathscr{A}}{\kappa}$ (which can be interpreted as a valuation on $\mathscr{A}$) to its center on ${\bbar{\mathfrak{X}}_s}$.
\end{constr}

\begin{lemma} \label{lemma:min_gen_rz}
Let $\mathscr{A}$ be a topologically finite type $\Oh_K$-algebra and let $A = \mathscr{A} \otimes_{\Oh_K} K$ and $\mathscr{A}_\kappa = \mathscr{A} \otimes_{\Oh_K} \kappa$. Let $x \in \Spa{A}{A^{\text{min}}}$ and $x_0$ be its minimal vertical generalization in $\Spa{A}{A^\circ}$ and let $R_x \subset \kappa(x_0)$ be the corresponding valuation ring. Let $\eta: \Spec{\mathscr{A}_\kappa} \to \RZ{\mathscr{A}_\kappa}{\kappa}$ be the natural map discussed in Section \ref{section:rz} and let $\iota: \kappa(\spmap(x_0)) \to  \kappa(x_0)$ be the natural homomorphism given by Remark \ref{rmk:iota}. Then: 
\begin{enumerate}[(i)]
    \item $\spmap_{\RZop}(x_0) = \eta(\spmap(x_0))$ and the diagram below commutes,
        \begin{center}
            \begin{tikzcd}
                \Spa{A}{A^\circ} \arrow[]{d}{\spmap} \arrow[hook]{r}{j} &
                \Spa{A}{A^{\text{min}}} \arrow[]{d}{\spmap_{\RZop}} 
                \\
                \Spec{\mathscr{A}_\kappa} \arrow[]{r}{\eta} &
                \RZ{\mathscr{A}_\kappa}{\kappa}
            \end{tikzcd}
        \end{center}
    \item $\spmap_{\RZop}(x_0)$ is the generalization of $\spmap_{\RZop}(x)$ that is minimal among those that lie inside~$\eta(\Spec{\mathscr{A}_\kappa})$,
    \item $\spmap_{\RZop}(x) \in \RZ{\mathscr{A}_\kappa}{\kappa}$ corresponds to the valuation supported in $\spmap(x_0)$ and given by the valuation ring~$\iota\inv(R_x)$.
\end{enumerate}
\end{lemma}
\begin{proof}
\textit{(i)} Let $x \in \Spa{A}{A^\circ}$. Consider $x$ as an element of $\Spa{\mathscr{A}}{\mathscr{A}^{\min}}$. Then the valuation $v_x : \mathscr{A} \to \Gamma_x$ has a trivial characteristic group, as $v_x(A^\circ) \leq 1$ and $\mathscr{A} \subset A^\circ$. Hence the retraction $\spmap_{RZ}(x)$ is a trivial valuation on the support of $x$. This is exactly the same valuation as $\eta(\spmap(x))$.

\textit{(ii)} Now let $x \in \Spa{A}{A^{\min}}$ and assume $x_0$ is its minimal vertical generalization in $\Spa{A}{A^\circ}$. Consider $x$ and $x_0$ as elements of $\Spa{\mathscr{A}}{\mathscr{A}^{\min}}$ and take the associated valuations $v_x: \mathscr{A} \to \Gamma_x$ and $v_{x_0}: \mathscr{A} \to \Gamma_{x_0}$. Then, by the construction of vertical generalizations, there exists a convex subgroup $H \subset \Gamma_x$ such that $v_{x_0} = v_x / H$, i.e. we have an isomorphism $\Gamma_{x_0} \simeq \Gamma_x / H$ and the valuation $v_{x_0}$ is given by the composition $v_{x_0} = \pi \circ v_x$, where $\pi$ is the natural quotient map $\pi: \Gamma_x \to \Gamma_x / H$. All vertical generalizations of a continuous and analytic valuation are continuous and analytic, so for any $H' \subset \Gamma_x$ the valuation $v_x / H'$ is continuous and analytic. Ordering by inclusion $H_1 \subseteq H_2 \subset \Gamma_x$ corresponds to ordering by specialization $v_x / H_1 \preceq v_x / H_2$. Thus $H$ is minimal among convex subgroups of $\Gamma_x$ satisfying $(v_x / H )(\mathscr{A}) \leq 1$. Thus $c\Gamma_x \subset H$ so simply taking $H = c\Gamma_x$ gives us minimal $H$ satisfying $(v_x / H )(\mathscr{A}) \leq 1$. Hence $v_{x_0} = v_x / c\Gamma_x$.

Let us now also regard $\spmap_{RZ}(x)$ and $\spmap_{RZ}(x_0)$ as elements $\ttilde{x}$ and $\ttilde{x_0}$ of $\Spa{\mathscr{A}}{\mathscr{A}^{\min}}$ and denote the corresponding valuations as $\ttilde{v}_x : \mathscr{A} \to \Gamma_x$ and $\ttilde{v}_{x_0} : \mathscr{A} \to \Gamma_{x_0}$. Observe that $\supp{\ttilde{v}_x} = \{ a \in \mathscr{A} \mid v_x(a) < 1, v_x(a) \not\in c\Gamma_x\}$ and $\supp{\ttilde{v}_{x_0}} = \{ a \in \mathscr{A} \mid v_{x_0}(a) < 1\}$, since $c\Gamma_{x_0} = \{1\}$. One easily verifies that $\supp{\ttilde{v}_{x}} = \supp{\ttilde{v}_{x_0}}$, as this amount to checking that for each $a \in \mathscr{A}$ the inequality $v_{x_0}(a) < a$ holds if and only if $v_x(a) < 1$ and $v_x(a) \not\in c\Gamma_x$. Moreover, $\ttilde{v}_{x_0}$ is trivial on this support. By Remark \ref{rmk:rz} this proves \textit{(ii)}.

\textit{(iii)} We have the maps $\sA \to k(x)^+ \to k(x_0)^+$. Observe that $k(x)^+$ and $k(x_0)^+$ are valuation rings and the maps from $\sA$ to them determine the associated valuations $x$ and $x_0$ on $\sA$. Then the valuation $\ttilde{x_0}$ is given by the composition $\sA \to k(x_0)^+ \to \kappa(x_0)$, where $\kappa(x_0)$ is considered as a trivial valuation ring. Now $\ttilde{x}$ is given by the map $\sigma$ below
\begin{center}
    \begin{tikzcd}
        \sA \arrow[]{r}{} \arrow[bend right=10]{dr}{\sigma}&
        k(x)^+ \arrow[]{r}{} \arrow[]{d}{}&
        k(x_0)^+ \arrow[]{d}{}\\
        & R_x \arrow[]{r}{} &
        \kappa(x_0)
    \end{tikzcd}
\end{center}
Observe that $\sA \to \kappa(x_0)$ factorizes through $\fracfield(\sA / \text{ker}(
\sA \to \kappa(x_0))) = \kappa(\spmap(x_0))$. Hence $\ttilde{x}$ can be also given by the valuation ring $\iota\inv(R_x)$. Of course the map $\sA \to \iota\inv(R_x)$ factorizes naturally through $\sA_\kappa$.
\end{proof}

\begin{cor}\label{cor:equal_maps}
Let $X = \Spa{A}{A^\circ}$ be an affinoid rigid space and let $(\mathfrak{X}, \bbar{\mathfrak{X}}_s) \in \mathcal{M}(\bbar{X})$. Then the maps of sets $\bbar{\spmap}, \bbar{\spmap}': \bbar{X} \to \bbar{\mathfrak{X}}_s$ (defined in Construction \ref{const:map} and Construction \ref{const:sp_prime} respectively) are equal.
\end{cor}
\begin{proof}
This is a straightforward corollary of \textit{(iii)} in the previous Lemma.
\end{proof}

\begin{lemma}[Naturality of the map $\bbar{\spmap}$] \label{l:f_restr}
Let $f: X \to Y$ be a morphism of quasi-compact separated rigid spaces and let $g: (\mathfrak{X}, \bbar{\mathfrak{X}}_s) \to (\mathfrak{Y}, \bbar{\mathfrak{Y}}_s)$ be a morphism of compactified formal schemes, such that the underlying map $g: \mathfrak{X} \to \mathfrak{Y}$ is a model of $f$. Then the following diagram commutes:
\begin{center}
    \begin{tikzcd}
    \bbar{X} \arrow[]{r}{\bbar{f}} \arrow[]{d}{\bbar{\spmap}_X} & \bbar{Y} \arrow[]{d}{\bbar{\spmap}_Y} \\
    \bbar{\mathfrak{X}}_s \arrow[]{r}{g} & \bbar{\mathfrak{Y}}_s
    \end{tikzcd}
\end{center}
where the map $\bbar{f}$ is obtained from the universal property of Huber's compactifications.
\end{lemma}
\begin{proof}
Let $x \in \bbar{X}$ and $y = \bbar{f}(x)$. Let $x_0$ and $y_0$ be their minimal generalizations. Then from \cite[Lemma 1.1.10.iv]{HuEt} we conclude that $\bbar{f}(x_0) = y_0$. Let $R_x \subset \kappa(x_0)$ and $R_y \subset \kappa(y_0)$ be valuations rings associated to $x$ and $y$ respectively. The map $f$ induces a homomorphism $\iota_f: \kappa(y_0) \to \kappa(x_0)$. From the definition of a morphism of adic spaces, we get that $\iota_f\inv(R_x) = R_y$. Now note that $g(\spmap(x_0)) = \spmap(y_0)$ and $g$ induces a homomorphism $\kappa(\spmap(y_0) \to \kappa(\spmap(x_0))$, which from obvious reasons gives us a commutative diagram 
\begin{center}
    \begin{tikzcd}
    \kappa(x_0) & \kappa(y_0) \arrow[swap]{l}{\iota_f}\\
    \kappa(\spmap(x_0)) \arrow[]{u}{}& \arrow[]{l}{} \arrow[]{u}{}\kappa(\spmap(y_0))
    \end{tikzcd}
\end{center}
By pulling back valuation rings on this diagram it is easy to conclude that $g$ maps the center of $R_x$ to the center of $R_y$.
\end{proof}

\begin{lemma} \label{f_properties}
For every quasi-compact separated rigid space $X$ and $(\mathfrak{X}, \bbar{\mathfrak{X}}_s) \in \mathcal{M}(\bbar{X})$ the map 
\[
    \bbar{\spmap}: \bbar{X} \to \overline{\mathfrak{X}_s}
\]
is continuous. Moreover, $\bbar{\spmap}$ is spectral, closed and surjective.
\end{lemma}
\begin{proof}
Choose a finite cover by affine formal schemes $\mathfrak{X} = \bigcup^{n}_{i=0} \mathfrak{U_i}$ for $\mathfrak{U_i} = \Spf{\mathscr{A}_i}$ and the associated cover $X = \bigcup_{i=1}^n U_i$ for $U_i = \Spa{A_i}{A_i^\circ}$, where $A_i \simeq \mathscr{A}_i \otimes_{\Oh_K} K$. Let $\overline{\mathfrak{U}_{i, s}}$ be the schematic closures of respective affine open subsets in $\overline{\mathfrak{X}_s}$. By Lemma \ref{l:f_restr} the following diagram commutes
\begin{center}
    \begin{tikzcd}
        \bigsqcup_{i=1}^n \bbar{U_i} \arrow[]{r}{i} \arrow[swap]{d}{f' = \bigsqcup \bbar{\spmap}_{U_i}}
        & \bbar{X} \arrow[]{d}{\bbar{\spmap}} \\
        \bigsqcup_{i=1}^n \overline{\mathfrak{U}_{i, s}} \arrow[]{r}{\widetilde{i}} & \overline{\mathfrak{X}_s}
    \end{tikzcd}
\end{center}
Here $i$ is the induced map on compactifications and $\ttilde{i}$ is the natural closed immersion. Note that $i$ is a proper map, since proper maps of adic spaces satisfy the cancellation property (this works exactly as \cite[Proposition 10.1.19]{Vakil} does for schemes). Hence it is closed. Moreover, it is surjective, since it is closed and contains the dense open subset $X \subset \bbar{X}$. It is thus a topological quotient map. Observe that this diagram commutes on the level of sets and $f', \widetilde{i}$ are continuous. Now, by the universal property of the topological quotient, we conclude that $\bbar{\spmap}$ is indeed continuous. Since all three maps $i, \widetilde{i}, f'$ are surjections, $\bbar{\spmap}$ must be too. 

For closedness of $\bbar{\spmap}$ take $Z \subset \overline{X}$ to be an arbitrary closed subset. Then $\bbar{\spmap}(Z) = \widetilde{i}(f'(i\inv(Z)))$ and the right hand side is closed as $f'$ and $\ttilde{i}$ are closed maps. Spectrality of $\bbar{\spmap}$ is proven similarly.
\end{proof}

Let us fix the following notation $x_\alpha := \pr_\alpha(x)$, where $\pr_\alpha: \ttilde{X} \to \bbar{\mathfrak{X}}_{\alpha, s}$ for some $(\mathfrak{X}_\alpha, \bbar{\mathfrak{X}}_{\alpha, s}) \in \mathcal{M}(\bbar{X})$. In what follows $\alpha$ will always denote an index of an element of $\mathcal{M}(\bbar{X})$.

\begin{lemma}\label{l:rz}
Let $x \in X$ and let $Z$ be its closure $Z = \overline{\{ x\}} \subset \bbar{X}$ and $\ttilde{Z}$ the closure $\ttilde{Z} = \overline{\{ \varphi(x)\}} \subset \ttilde{X}$. Then $\varphi|_Z: Z \to \ttilde{Z}$ is a bijection.
\end{lemma}
\begin{proof}
With a slight abuse of notation, we write $x \in X \subset \ttilde{X}$ instead of $\varphi(x)$.  Firstly, we will prove that $\ttilde{Z} = \lim\limits_{\longleftarrow} \overline{\{ x_\alpha\}}$, where the index $\alpha$ runs over all elements of $\mathcal{M}(\bbar{X})$. Let $y = (y_\alpha)\in RHS$. This means that for all $\alpha$ the point $y_\alpha$ is a specialization of $x_\alpha$. If $y \not\in LHS$ then there exists a quasi-compact open $U \subset \ttilde{X}$ with $y \in U$ and $x \not\in U$. By Proposition \ref{p:spectral_limit} there exists $\alpha'$ together with a quasi-compact open $V \subset \overline{\mathfrak{X}_{\alpha', s}}$ such that $U = \pr_{\alpha'}\inv(V)$. But $x_\alpha \in V$, hence $x \in U$, contradiction. Thus $RHS \subset LHS$. The other inclusion is obvious.

We will use notation $Z_\alpha = \overline{\{ x_\alpha\}} \subset \overline{\mathfrak{X}_{\alpha, s}}$. We shall now prove that there is a bijection $\lim\limits_{\longleftarrow} Z_\alpha \simeq \RZ{\kappa(x)}{\kappa}$. Recall that $\RZ{\kappa(x)}{\kappa}$ is by definition an inverse limit over all integral proper schemes $Y_i$ over $\kappa$ with an inclusion $k(Y_i) \subseteq \kappa(x)$. 
See that by Construction \ref{const:spec} we have that $\kappa(x) \simeq \lim\limits_{\longrightarrow} k(Z_\alpha)$. By \cite[Lemma 3.11]{ALY} if we pick any $Z_{\alpha_0}$ then the system $(Z_\alpha)$ contains all schematic blow-ups of $Z_{\alpha_0}$. Thus $\lim\limits_{\longleftarrow} (Z_\alpha) = \lim\limits_{\longleftarrow}\RZ{k(Z_\alpha)}{\kappa}$. And since giving a valuation on $\kappa(x) \simeq \lim\limits_{\longrightarrow} k(Z_\alpha)$ amounts to giving a valuation on each level of the limit we get that $\RZ{\kappa(x)}{\kappa} \simeq \lim\limits_{\longleftarrow} \RZ{k(Z_\alpha)}{\kappa}$. This gives us the bijection $\ttilde{Z} \simeq \RZ{\kappa(x)}{\kappa}$.

Since $\bbar{X}$ is proper over $\Spa{K}{\Oh_K}$ then by Proposition \ref{prop:rz_as_closure} we get a natural bijection $Z \simeq \RZ{\kappa(x)}{\kappa}$. We easily check that the bijections $Z \simeq \RZ{\kappa(x)}{\kappa}$ and $\ttilde{Z} \simeq \RZ{\kappa(x)}{\kappa}$ are compatible with $\varphi$.
\end{proof}

\begin{cor}
The map $\varphi: \bbar{X} \to \ttilde{X}$ is bijective.
\end{cor}
\begin{proof}
First we prove surjectivity of $\varphi$. Note, that it is enough to prove that $X \subset \ttilde{X}$ is pro-constructible dense. In that case by Proposition \ref{prop:spectral} each point $x \in \ttilde{X}$ has a generalization $x_0 \in X$. Using Lemma \ref{l:rz} we conclude that all specializations of $x_0$ are in the image of $\varphi$.

Note that for each $\alpha$ we have a quasi-compact open immersion $\mathfrak{X}_{\alpha, s} \subseteq\overline{\mathfrak{X}_{\alpha, s}}$. Moreover $X = \lim\limits_{\longleftarrow} \mathfrak{X}_{\alpha, s}$, so $\ttilde{X} \backslash X = \bigcup_\alpha \pr_\alpha\inv(\overline{\mathfrak{X}_{\alpha, s}} \backslash \mathfrak{X}_{\alpha, s})$, which is open in constructible topology. Hence $X$ is pro-constructible in $\ttilde{X}$. Suppose that $X$ is not dense and $y$ is not in its closure. 
Then there exists a quasi-compact open $U \subset \ttilde{X}$ such that $U \cap X = \emptyset$ and $y \in U$. 
By Proposition \ref{p:spectral_limit} we can find $\alpha'$ together with a quasi-compact open $V \in \overline{\mathfrak{X}_{\alpha', s}}$ such that $U = \pr_{\alpha'}\inv(V)$. But $V \cap \mathfrak{X}_{\alpha', s} \not= \emptyset$ as $\mathfrak{X}_{\alpha', s}$ is dense in $\overline{\mathfrak{X}_{\alpha', s}}$. So $U \cap X \not= \emptyset$, which is a contradiction. Hence $X$ is dense.

Now we prove injectivity of $\varphi$. Suppose that $x, y \in \bbar{X}$ and $\varphi(x) = \varphi(y) = z$. Let $x', y'$ be maximal generalizations of $x$ and $y$, respectively. If $x' = y'$ then we have contradiction with Lemma \ref{l:rz}. 

Assume $x' \not= y'$. Consider the systems $(x'_\alpha)$ and $(y'_\alpha)$. There exists $\alpha_0$ such that $x'_{\alpha_0} \not= y'_{\alpha_0}$. Let $Z_1$ and $Z_2$ be closures of $x'_{\alpha_0}$ and $y'_{\alpha_0}$ in $\mathfrak{X}_s$. Let $Y$ be the schematic blow-up of $\mathfrak{X}_s$ in $Z_1 \cap Z_2$. Then the strict transforms $\ttilde{Z_1}, \ttilde{Z_2}$ of $Z_1$ and $Z_2$ to $Y$ are disjoint by \cite[Exercise II.7.12]{Ha}.  Now by \cite[Lemma 3.11]{ALY} we can find an admissible blow-up $\mathfrak{X}$ of $\mathfrak{X}_{\alpha_0}$ such that $\mathfrak{X}_s \simeq Y$. Let $\overline{Y}$ be any compactification of $Y$. See that $(\mathfrak{X}, \overline{Y}) \in \mathcal{M}(\bbar{X})$. If now the closures of $\ttilde{Z_1}$ and  $\ttilde{Z_2}$ in $\bbar{Y}$ intersect at $W$, then $W \subset \bbar{Y} \backslash Y$. Let $\bbar{Y_0}$ be the blow-up of $\bbar{Y}$ in $W$. Observe that $\bbar{Y_0}$ is still a $\kappa$-compactification of $Y$ and moreover strict transforms of the closures of $\ttilde{Z_1}$ and  $\ttilde{Z_2}$ are disjoint. So projections of $\varphi(x)$ and $\varphi(y)$ on $\bbar{Y_0}$ are different, so in particular $\varphi(x) \not= \varphi(y)$.   
\end{proof}

\begin{proof}[Proof of Proposition \ref{prop:homeo}]
We already know that $\varphi: \bbar{X} \to \ttilde{X}$ is a bijective continuous function, so it is enough to prove that it is a closed map. First however we need to prove that $\varphi$ is spectral. Let us take a quasi-compact open $U \subset \ttilde{X}$. Then by Proposition \ref{p:spectral_limit} there exists $\overline{\mathfrak{X}_{\alpha, s}}$, with projection $\pr_\alpha: \ttilde{X} \to \overline{\mathfrak{X}_{\alpha, s}}$ together with a quasi-compact open $V \subset \overline{\mathfrak{X}_{\alpha, s}}$ such that $U = \pr_\alpha\inv(V)$. Consider the diagram.
\begin{center}
    \begin{tikzcd}
    \bbar{X} \arrow[swap]{dr}{\bbar{\spmap}_\alpha} \arrow[]{r}{\varphi} &
    \ttilde{X} \arrow[]{d}{\pr_\alpha} \\
    & \overline{\mathfrak{X}_{\alpha, s}}
    \end{tikzcd}
\end{center}
Now $\varphi\inv(U) = \bbar{\spmap}_\alpha\inv(V)$ and the latter is quasi-compact, as $\bbar{\spmap}_\alpha$ is spectral by Lemma \ref{f_properties}. Hence $\varphi$ is spectral.

Now, take a closed subset $Z \subset \bbar{X}$. Since $\varphi$ is spectral the set $\varphi(Z)$ is pro-constructible by Proposition \ref{prop:spectral}. Thus by Proposition \ref{p:spectral_limit} we have equality $\varphi(Z) = \bigcap_\alpha \pr_\alpha\inv(\bbar{\spmap}_\alpha(Z))$. See that $RHS$ is closed since we already proved that $\bbar{\spmap}_\alpha$ is closed in Proposition \ref{f_properties}. Hence $\varphi$ is a continuous closed bijection, thus a homeomorphism. 
\end{proof}

We shall finish this section with the following proposition that will be useful when considering sheaves.
\begin{prop}\label{prop:sheaf_tilde_rz}
We have a natural homeomorphism
$$
\ttilde{X} = \lim\limits_{\substack{\longleftarrow \\ (\mathfrak{X}, \overline{\mathfrak{X}_{s}}) \in \mathcal{M}(\overline{X})}} \overline{\mathfrak{X}_{s}} \quad \simeq \quad 
\lim\limits_{\substack{\longleftarrow \\ \mathfrak{X} \in \mathcal{M}(X)}} \RZ{\mathfrak{X}_s}{\kappa}
$$
\end{prop}
\begin{proof}
This is an exercise in category theory, which is left to the reader. It is straightforward to define a map either way and then prove that the $RHS$ satisfies the same universal property as the $LHS$ (given that $\RZ{\mathfrak{X}_s}{\kappa}$ is the inverse limit of all $\kappa$-compactifications of $\mathfrak{X}_s$).
\end{proof}

\subsection{Comparison of structure sheaves}
We shall now compare the sheaves $\Oh_{\ttilde{X}}$ on $\ttilde{X}$ and $\Oh_{\bbar{X}}^+$ on $\bbar{X}$. First, we observe that one can construct $\Oh_{\ttilde{X}}$ using Riemann--Zariski spaces. Pick $\mathfrak{X} \in \mathcal{M}(X)$ and consider the space $\RZ{\mathfrak{X}_s}{\kappa}$. Then by Corollary \ref{rz_comp} we have $\RZ{\mathfrak{X}_s}{\kappa} = \varprojlim \overline{\mathfrak{X}_s}$, where the limit goes over all $\kappa$-compactifications of $\mathfrak{X}_s$. Now we define a sheaf $\Oh_{\overline{\mathfrak{X}}}$ on $\RZ{\mathfrak{X}_s}{\kappa}$ to be the following pullback
\begin{center}
    \begin{tikzcd}
    \Oh_{\overline{\mathfrak{X}}}\arrow[]{r}{} \arrow[]{d}{u} \arrow[dr, phantom, "\square"]
    & \Oh_{\RZ{\mathfrak{X}_s}{\kappa}} \arrow[]{d}{}\\
    \eta_* \Oh_{\mathfrak{X}} \arrow[]{r}{}& \eta_* \Oh_{\mathfrak{X}_s}
    \end{tikzcd}
\end{center}
where $\eta:\mathfrak{X}_s \to \RZ{\mathfrak{X}_s}{\kappa}$ is the natural map as in Section \ref{section:rz}.

For any $(\mathfrak{X}_i, \bbar{\mathfrak{X}}_{ij, s}) \in \mathcal{M}(\bbar{X})$ let $\pr_{ij}: \ttilde{X} \to \bbar{\mathfrak{X}}_{ij, s}$ and $r_i: \ttilde{X} \to \RZ{\mathfrak{X}_{i, s}}{\kappa}$ be the natural projections. Moreover let $p_{ij}: \RZ{\mathfrak{X}_{i, s}}{\kappa} \to \bbar{\mathfrak{X}}_{ij, s}$ also be the natural projection. Observe that $\pr_{ij} = p_{ij} \circ r_i$. Here index $i$ runs over all admissible formal models of $X$. For a fixed $i$ letter $j$ indexes all $\kappa$-compactifications of $\mathfrak{X}_{i,s}$. Now, we need the following lemma.

\begin{lemma} \label{lemma:rz_sheaf}
There is a natural isomorphism of sheaves on $\ttilde{X}$
$$
\Oh_{\ttilde{X}} = 
\lim\limits_{\substack{\longrightarrow \\ (\mathfrak{X}_i, \bbar{\mathfrak{X}}_{ij, s}) \in \mathcal{M}(\bbar{X})}}
\pr_{ij}\inv \Oh_{(\mathfrak{X}_i, \bbar{\mathfrak{X}}_{ij, s})} 
\simeq 
\lim\limits_{\substack{\longrightarrow \\ \mathfrak{X}_i \in \mathcal{M}(X)}} 
r_i \inv \Oh_{\overline{\mathfrak{X}}_i}.
$$
\end{lemma}
\begin{proof}
This is a purely formal consideration based on the properties of limits and colimits of sheaves. First, let us look at the limit in the $LHS$. Note that inverse image functor $\pr_{ij}\inv$ commutes with finite limits. So we have the following Cartesian square
\begin{center}
    \begin{tikzcd}
    \pr_{ij}\inv \Oh_{(\mathfrak{X}_i, \bbar{\mathfrak{X}}_{ij, s})}  \arrow[]{r}{} \arrow[hook]{d}{} \arrow[dr, phantom, "\square"]
    & \pr_{ij}\inv \Oh_{\bbar{\mathfrak{X}}_{ij, s}} \arrow[hook]{d}{}\\
    \pr_{ij}\inv (h_{ij})_* \Oh_{\mathfrak{X}} \arrow[]{r}{}& \pr_{ij}\inv (h_{ij})_* \Oh_{\mathfrak{X}_s}
    \end{tikzcd}
\end{center}
where $h_{ij}: \mathfrak{X}_{i, s} \hookrightarrow \bbar{\mathfrak{X}}_{ij, s}$ is the open immersion that is part of the data of $(\mathfrak{X}_i, \bbar{\mathfrak{X}}_{ij, s})$. Moreover vertical arrows are injective, as the inverse image functor is exact. Now, since filtered colimits commute with finite limits we get another Cartesian square.
\begin{center}
    \begin{tikzcd}
    \varinjlim \pr_{ij}\inv \Oh_{(\mathfrak{X}_i, \bbar{\mathfrak{X}}_{ij, s})}  \arrow[]{r}{} \arrow[hook]{d}{} \arrow[dr, phantom, "\square"]
    & \varinjlim \pr_{ij}\inv \Oh_{\bbar{\mathfrak{X}}_{ij, s}} \arrow[hook]{d}{}\\
    \varinjlim \pr_{ij}\inv (h_{ij})_* \Oh_{\mathfrak{X}_i} \arrow[]{r}{}
    & \varinjlim \pr_{ij}\inv (h_{ij})_* \Oh_{\mathfrak{X}_{i, s}}
    \end{tikzcd}
\end{center}
Moreover vertical arrows are injective, as the colimit functor is exact.
After unwrapping the $RHS$ in an analogous way we get
\begin{center}
    \begin{tikzcd}
    \varinjlim r_i \inv \Oh_{\overline{\mathfrak{X}}_i}\arrow[]{r}{} \arrow[]{d}{u} \arrow[dr, phantom, "\square"]
    & \varinjlim r_i \inv \Oh_{\RZ{\mathfrak{X}_{i,s}}{\kappa}} \arrow[]{d}{}\\
    \varinjlim r_i \inv (\eta_i)_* \Oh_{\mathfrak{X}_i} \arrow[]{r}{}
    & \varinjlim r_i \inv (\eta_i)_* \Oh_{\mathfrak{X}_{i, s}}
    \end{tikzcd}
\end{center}
where $\eta_i: \mathfrak{X}_{i, s} \to \RZ{\mathfrak{X}_{i, s}}{\kappa}$ is the obvious map. Since both $\eta_i$ and $h_{ij}$ are injective on points, sheaves $r_i \inv (\eta_i)_* \Oh_{\mathfrak{X}_i}$ and $\pr_{ij}\inv (h_{ij})_* \Oh_{\mathfrak{X}_i}$ are naturally isomorphic (look at stalks). Analogously for $r_i \inv (\eta_i)_* \Oh_{\mathfrak{X}_{i, s}}$ and $\pr_{ij}\inv (h_{ij})_* \Oh_{\mathfrak{X}_{i, s}}$. So in order to finish the prove it suffices to show that $\varinjlim r_i \inv \Oh_{\RZ{\mathfrak{X}_{i, s}}{\kappa}}$ and $\varinjlim \pr_{ij}\inv \Oh_{\bbar{\mathfrak{X}}_{ij, s}}$ are naturally isomorphic. But from construction we have
$$
\lim\limits_{\substack{\longrightarrow \\ i}} r_i \inv \Oh_{\RZ{\mathfrak{X}_{i, s}}{\kappa}}
\simeq
\lim\limits_{\substack{\longrightarrow \\ i}} r_i\inv \lim\limits_{\substack{\longrightarrow \\ j}}
p_{ij}\inv \Oh_{\bbar{\mathfrak{X}}_{ij, s}}
\simeq
\lim\limits_{\substack{\longrightarrow \\ i}} \lim\limits_{\substack{\longrightarrow \\ j}} \pr_{ij}\inv \Oh_{\bbar{\mathfrak{X}}_{ij, s}}
$$
Here the second isomorphism is inverse image commuting with colimit (as colimits commute with colimits). Now it is enough to show that
$$
\lim\limits_{\substack{\longrightarrow \\ i, j}} \pr_{ij}\inv \Oh_{\bbar{\mathfrak{X}}_{ij, s}}
\simeq 
\lim\limits_{\substack{\longrightarrow \\ i}} \lim\limits_{\substack{\longrightarrow \\ j}} \pr_{ij}\inv \Oh_{\bbar{\mathfrak{X}}_{ij, s}}
$$
which is straightforward and goes in a similar way to the proof of Proposition \ref{prop:sheaf_tilde_rz}.
\end{proof}

We observe that the sheaf $\varphi\inv\Oh_{\ttilde{X}}$ is naturally a subsheaf of $j_*\Oh_X^+$. Indeed, in the proof of the previous lemma we had a natural inclusion
\begin{center}
    \begin{tikzcd}
    \Oh_{\ttilde{X}} = \varinjlim \pr_{ij}\inv \Oh_{(\mathfrak{X}_i, \bbar{\mathfrak{X}}_{ij, s})} \arrow[hook]{r}{}
    &
    \varinjlim \pr_{ij}\inv (h_{ij})_* \Oh_{\mathfrak{X}_i}.
    \end{tikzcd}
\end{center}
So it is enough to prove that the $RHS$ is canonically isomorphic to $j_*\Oh_X^+$. And this is again an easy check on stalks.

Let $j: X \hookrightarrow \bbar{X}$ be the obvious open immersion. Then there is a natural map $\Oh_{\overline{X}}^+ \to j_* \Oh_X^+$. We will show that this map is injective and identify $\Oh_{\overline{X}}^+$ and $\varphi\inv\Oh_{\ttilde{X}}$ as the same subsheaf of $j_*\Oh_X^+$.

\begin{prop}\label{prop:iso_sheaves}
The sheaves $\varphi\inv\Oh_{\ttilde{X}}$ and $\Oh_{\overline{X}}^+$ are equal as subsheaves of $j_*\Oh_X^+$.
\end{prop}
\begin{proof}
Let $x \in \overline{X}$ and $x_0 \in X$ such that $x_0$ is the minimal vertical generalization of $x$. By $R_x\subset \kappa(x_0)$ we denote the valuation ring associated to $x$ and by $\pi: \Oh_{X, x_0}^+ \to \kappa(x_0)$ the natural quotient map.

We first prove that $(j_*\Oh_X^+)_x \simeq (j_*\Oh_X^+)_{x_0}$. By the definition of a stalk of a sheaf it is enough to prove that for any open neighborhood $U \subset X$ of $x_0$ there exists open neighborhood $V \subset \bbar{X}$ of $x$ such that $V \cap X \subset U$. Let $U \subset X$ be a quasi-compact open neighborhood of $x_0$. Then $X \ \backslash \ U$ is a pro-constructible subset of $X$ and hence it is also pro-constructible in $\bbar{X}$ (because $j$ is spectral by Theorem \ref{thm:comp}). Thus (by Proposition \ref{prop:spectral}) taking the topological closure $\bbar{X \ \backslash \ U}$ in $\bbar{X}$ amounts to taking all specializations. Since $x_0$ is the minimal generalization of $x$ in $X$ we get that $x \not\in \bbar{X \ \backslash \ U}$. Then $V = \bbar{X} \ \backslash \ \bbar{X \ \backslash \ U}$ is the open neighborhood of $x$ that we need.

Moreover, we have $(j_*\Oh_X^+)_x \simeq (j_*\Oh_X^+)_{x_0} \simeq \Oh_{X, x_0}^+ \simeq \Oh_{\bbar{X}, x_0}^+$. Hence the map $\Oh_{\bbar{X}, x}^+ \to (j_*\Oh_X^+)_x \simeq \Oh_{\bbar{X}, x_0}^+$ is injective by \cite[Lemma 1.1.10]{HuEt}. So $\Oh_{\overline{X}}^+$ is indeed a subsheaf of $j_*\Oh_X^+$.
Analogously $\Oh_{\overline{X}, x} \simeq (j_*\Oh_X)_x \simeq (j_*\Oh_X)_{x_0} \simeq \Oh_{X, x_0}$. Observe that this gives us inclusions $\Oh_{\overline{X}, x}^+ \subset \Oh_{X, x_0}^+ \subset \Oh_{\overline{X}, x}$. Thus according to \cite[Proposition III.6.3.1]{Morel} we have the description
\begin{gather*}
\Oh_{\overline{X}, x}^+ = \left\{ f \in \Oh_{X, x_0}^+ : |f(x)| \leq 1 \right\} = \left\{ f \in \Oh_{X, x_0}^+ : \pi(f) \in R_x \right\} = \pi\inv(R_x).
\end{gather*}

We will now show that the same statement holds for $\varphi\inv\Oh_{\ttilde{X}, x}$. Since we have already proved that $\varphi$ is a homeomorphism, we can look at the stalks of $\Oh_{\ttilde{X}, x}$ and forget about the inverse image functor $\varphi$. First, let us pick any formal model $\mathfrak{X} \in \mathcal{M}(X)$. Let $\spmap_{\RZop}: \overline{X} \to \RZ{\mathfrak{X}_s}{\kappa}$ be as in Construction \ref{const:phi} and $\eta:\mathfrak{X}_s \to \RZ{\mathfrak{X}_s}{\kappa}$ the natural immersion. By Proposition \ref{rz_sheaf} $\spmap_{\RZop}(x)$ possesses a minimal generalization $y$ in $\eta(\mathfrak{X}_s)$. By Lemma \ref{lemma:min_gen_rz} we get that $y = \spmap(x_0)$.

Recall that points on $\RZ{\mathfrak{X}_s}{\kappa}$ are certain valuations on $\mathfrak{X}_s$ with center on $\Spec{\kappa}$. According to Construction \ref{const:map} and Corollary \ref{cor:equal_maps} the point $\spmap_{\RZop}(x)$ is associated to the valuation on the following diagram
\begin{center}
    \begin{tikzcd}
    \Spec{\kappa(x_0)} \arrow[]{r}{} \arrow[]{d}{} & \Spec{\kappa(\spmap(x_0))}  \arrow[]{r}{} \arrow[]{d}{}
    & \mathfrak{X}_s \arrow[]{d}{} \\
    \Spec{R_x} \arrow[]{r}{} & \Spec{\ttilde{R_x}} \arrow[]{r}{} & \Spec{\kappa}
    \end{tikzcd}
\end{center}
where $\ttilde{R_x}$ is the preimage of $R_x$ in $\kappa(\spmap(x_0))$ and the left square is the equivalence of valuations. Note that minimality of $x_0$ is responsible for the above valuation being an element of $\RZ{\mathfrak{X}_s}{\kappa}$. Let $\pi': \Oh_{\mathfrak{X}_s, \spmap(x_0)} \to \kappa(\spmap(x_0))$. 
From Proposition \ref{rz_sheaf} we conclude that $\Oh_{\RZ{\mathfrak{X}_s}{\kappa}, x} = (\pi')\inv(\ttilde{R_x}) \subset \Oh_{\mathfrak{X}_s, \spmap(x_0)}$. Hence, if $\pi_1: \Oh_{\mathfrak{X}, \spmap(x_0)} \to \kappa(\spmap(x_0))$ is the natural quotient, then $\Oh_{\overline{\mathfrak{X}}, x} = (\pi_1)\inv(\ttilde{R_x})$. So we have the following pullback diagram for every $\mathfrak{X} \in \mathcal{M}(X)$
\begin{center}
    \begin{tikzcd}
    \Oh_{\overline{\mathfrak{X}}, \spmap_{\RZop}(x)} \arrow[hook]{r}{} \arrow[]{d}{\pi} \arrow[dr, phantom, "\square"]
    & \Oh_{\mathfrak{X}, \spmap(x_0)} \arrow[]{d}{}\\
    \ttilde{R_{x}} \arrow[hook]{r}{} & \kappa(\spmap(x_0))
    \end{tikzcd}
\end{center}
Let $\spmap_i: X \to \mathfrak{X}_i$ be the specialization map. Then we have
$$
\lim\limits_{\substack{\longrightarrow \\ \mathfrak{X}_i \in \mathcal{M}(X)}} \kappa(\spmap_i(x_0))
\simeq
\kappa(x_0).
$$
This is an easy corollary of Theorem \ref{thm:adic_as_limit}. After taking the colimit, arguing in the spirit of Lemma \ref{lemma:rz_sheaf} and using the above isomorphism we get that $\Oh_{\ttilde{X}, x} \simeq \pi\inv(R_x)$, which concludes the proof.
\end{proof}

\subsection{Closing remarks}
We would like to make a few comments on the generality of the proof and chosen methods.
\begin{enumerate}
    \item It is probably true that the considerations can be made relative, i.e. it should be possible to prove a similar theorem for a relative compactification of a morphism of rigid analytic spaces.
    \item We need the ``quasi-compact and separated'' assumption to use the Nagata theorem on compactifications. We have not thought about generalizations to quasi-paracompact spaces.
\end{enumerate}

\subsection{Relation to rigid spaces of Fujiwara and Kato}\label{section:agv}

In \cite{FK}, Fujiwara and Kato develop a very general theory of rigid spaces, extending Raynaud's theory of admissible blow-ups of formal schemes. As proved in \cite[Corollary 1.2.7]{AGV}, this framework is large enough to encompass many adic spaces which are not rigid spaces. More precisely, there exists a fully faithful functor from the category of uniform analytic adic spaces to the Fujiwara--Kato category of rigid spaces, sending an affinoid adic space $\Spa{A}{A^+}$ to $\Spf{A^+}_{\rm rig}$, the rigid generic fiber of the formal spectrum of $A^+$. Moreover, under this correspondence, the adic space is naturally identified with the inverse limit of all formal models of the corresponding rigid space. 

In particular, if $X$ is an adic space locally of finite type over $\Spa{K}{ \Oh_K}$ (simply called a ``rigid space'' above) which is reduced (so that it is uniform), then its Huber universal compactification $\overline{X}$ is again a uniform analytic adic space. Therefore, by the result of \cite{AGV}, it admits ``formal models'', whose inverse limit is homeomorphic to $\overline{X}$. This gives a statement sounding very similar to our main theorem. Let us explain why the results in fact differ in subtle ways.

As we mention in Remark~\ref{rmk:is_formal_scheme} and explain in Appendix~\ref{section:app}, the locally ringed space $\bXX = (\bXX_s, \Oh_{(\mathfrak{X}, \overline{\mathfrak{X}}_{s})})$ associated to a compactifed formal scheme $(\XX, \bXX_s) \in \mathcal{M}(\bbar{X})$ is also a formal scheme. However, the following issues arise:
\begin{enumerate}[(i)]
    \item in general, $\bXX$ is not one of the formal models of $\bbar{X}$ that are obtained in \cite{AGV},
    \item it is not clear whether the inverse limit of all admissible formal blow-ups of $\bXX$ is isomorphic to $\bbar{X}$,
    \item even if the previous statement were true, it would not imply that any two $\bXX_1, \bXX_2 \in \mathcal{M}(\bbar{X})$ would share a common blow-up, as these formal schemes are no longer Noetherian.
\end{enumerate}
We tried addressing \textit{(ii)} and \textit{(iii)}, but without success. However, we conjecture them to be true in the case of discretely valued $K$. In that situation all formal models of $X$ are Noetherian and it helps us avoid some issues with non-finite normalization (cf. \cite[\S A.5]{FK}).

Even with the above questions answered, key differences persist. Namely, our theorem provides a description of the universal compactification in terms of objects "of finite type": pairs consisting of a formal model of $X$ and a compactification of its special fiber. In contrast, formal models (in the sense of \cite{AGV}) of the rigid space associated to the universal compactification are non-noetherian even if $K$ is discretely valued. We expect this difference to be important in applications. For example, if $X$ is smooth, using our perspective on $\overline{X}$, it should suffice to look at semistable (or log smooth) formal models of $X$, endowed with a log smooth (or good in the sense of \cite{LogConn}) compactification of their special fiber. We do not know which of the formal models of the rigid space corresponding to $\overline{X}$ could play a similar role.




\subsection{Relation to specialization morphisms of Gaisin and Welliaveetil}\label{section:gaisin}
In \cite{Gaisin} the authors develop a general theory of \textit{specialization morphisms} between analytic adic spaces and schemes (see \cite[\S 2]{Gaisin}). Namely, for an analytic adic space $X$ they consider the locally ringed space
$$
X_{\text{red}} = \left(X, \Oh_X^+ / \mathfrak{m}_{\Oh_X^+}\right),
$$
where $\mathfrak{m}_{\Oh_X^+} \subset \Oh_X^+$ is a subsheaf consisting of everywhere vanishing functions. Let $S$ be a scheme. Then any morphism of locally ringed spaces $X_{\text{red}} \to S$ is called a \textit{specialization morphism} and is denoted by an arrow $X \to S$.

In our context, for a sufficiently nice rigid analytic space $X$ over $\Spa{K}{\Oh_K}$ we can get specialization morphisms $f\colon X \to \Spec{k}$, as well as $g\colon X \to \bXX_{s}$ for any $\bXX_{s} \in \mathcal{M}(\bbar{X})$. Now by \cite[Theorem 3.30]{Gaisin} each of these specialization morphisms can be canonically compactified resulting in an analytic adic space for each one of them. It turns out that, for formal reasons, each of these compactifications is naturally isomorphic to Huber's compactification $\bbar{X}$ of $X$ over $\Spa{K}{\Oh_K}$. This gives us a different way of constructing a continuous map
$$
\varphi\colon \bbar{X} \to \varprojlim_{\bXX_s \in \mathcal{M}(\bbar{X})} \bXX_{_s}.
$$
Moreover, the specialization map $\bbar{X} \to \bXX_s$ is proper in the sense of \cite[Definition 3.18]{Gaisin}. It seems that with some work, our Lemma \ref{f_properties} can be deduced from this fact.

However, we did not find an argument within the theory of Gaisin--Welliaveetil to show that $\varphi$ is an isomorphism. It remains an interesting question whether one can prove it in a more abstract way than we do in this article.

\appendix

\section{Pushout of formal schemes} \label{section:app}

An affine formal scheme $\XX = \Spf{A}$ is called \textit{adic} if $A$ has $I$-adic topology and it is $I$-adically complete and separated.
If moreover there exists an ideal of definition $J \subset A$, which is finitely generated then $\XX$ is called \textit{of finite ideal type}. If $J$ is principal and $A$ is $J$-torsion-free, then $\XX$ is called \textit{of principal ideal type}. All of these notions generalize to arbitrary formal schemes, by looking at affine covers.

Let $\XX$ be an adic formal scheme with a finitely generated ideal of definition $\II \subset \Oh_\XX$. Then the scheme $\XX_{\rs} := \left( \XX \otimes (\Oh_\XX / \II ) \right)_{\red}$ is called the \textit{reduced special fiber} of $\XX$.

\begin{prop}\label{prop:is_formal_scheme}
    Let $\mathfrak{X}$ be a principal ideal type qcqs formal scheme, $\mathfrak{X}\urs$ its reduced special fiber and $j\colon \mathfrak{X}\urs \hookrightarrow Y$ an affine schematically dense open immersion, such that $Y \backslash \XX\urs$ is the support of a Cartier divisor in $Y$. Consider a sheaf of topological rings $\Oh_{\mathfrak{X}_Y} := j_* \Oh_{\mathfrak{X}} \times_{j_* \Oh_{\mathfrak{X}\urs}} \Oh_{Y}$ and denote by $\XX_Y$ the locally topologically ringed space $\mathfrak{X}_Y := (|Y|, \Oh_{\mathfrak{X}_Y})$. Then 
    \begin{enumerate}[(i)]
        \item $\XX_Y$ is a principal ideal type adic formal scheme with reduced special fiber $Y$,
        \item the induced map $\ttilde{j}\colon \XX \to \XX_Y$ is an affine open immersion,
        \item the diagram
            \begin{center}
                \begin{tikzcd}
                    \XX \arrow[]{r}{}&
                    \XX_Y \\
                    \XX\urs \arrow[]{r}{} \arrow[]{u}{}&
                    Y \arrow[]{u}{}
                \end{tikzcd}
            \end{center}
            is a push-out in the category of locally topologically ringed spaces.
    \end{enumerate}
\end{prop}

\begin{proof}
    We check everything affine-locally. By the assumptions, $Y$ is locally of the form $\Spec{B_0} = U \subset Y$, such that $j\inv (U) = V = \Spec{A_0}$ with $A_0 = B_0[h\inv]$ for some $h \in B_0$. If we take the open formal subscheme $\VV = \iota(V)$ of $\XX$, where $\iota\colon \XX\urs \to \XX$ is the natural closed immersion, then we can assume that $\VV = \Spf{A}$ for some $A$ with an ideal of definition $I = (\varpi)$ and $A_0 = (A / I)_\text{red}$. For simplicity, let us write $B := \Oh_{\XX_Y}(U)$. As a pullback of sheaves is already a sheaf (no need to sheafify) then $B$ is given by the pullback below.
    \begin{center}
        \begin{tikzcd}
            B \arrow[]{d}{\pi'} \arrow[]{r}{\ttilde{g}} &
            A \arrow[]{d}{\pi} \\
            B_0 \arrow[]{r}{g} &
            A_0
        \end{tikzcd}
    \end{center}
    This is equivalent to the following exact sequence of $B$-modules.
    \begin{equation}\label{eq:ses1}
        \begin{tikzcd}
            0 \arrow[]{r}{} &
            B \arrow[]{r}{\ttilde{g} \oplus \pi'} &
            A \oplus B_0 \arrow[]{r}{\pi - g} &
            A_0 \arrow[]{r}{} &
            0
        \end{tikzcd}
    \end{equation}
    
    \textit{(i)}
    Observe that $B$ is a subring of $A$ because $B_0$ is a subring of $A_0$. Since $\varpi \in B$ we can define a descending filtration $F^{\bullet} = \{ F^i \}_{i \in \mathbf{N}}$, with $F^i = \varpi^i B$. Then $F^\bullet$ defines the $\varpi$-adic topology on $B$. Let $G^\bullet = \{ G^i \}_{i \in \mathbf{N}}$, with $G^i = B \cap \varpi^i A$. Then $G^\bullet$ gives $B$ the induced topology of the subring of $A$. Obviously $F^i \subset G^i$. Moreover, $G^i = B \cap \varpi^{i-1} \left( \varpi A \right)$ and $\varpi A \subset B$, so $G^i \subset F^{i-1}$. Hence by \cite[Corollary 0.7.1.4]{FK} the topologies defined by $F^\bullet$ and $G^\bullet$ coincide. 
    Since $A, B_0$ and $A_0$ are complete, from the exact sequence above we conclude, that $B$ is complete with respect to the topology given by $G^\bullet$ (by \cite[Proposition 0.7.1.13]{FK}), hence it is also $\varpi$-adically complete. $B$ is $\varpi$-adically separated as $A$ is.

    We have thus so far proven that $B$ is an adic $\varpi$-adically complete ring. It is now enough to justify that $(|U|, \Oh_{\XX_Y}|_U) \simeq \Spf{B}$. We know that $(B / \varpi)_{\text{red}} \simeq B_0$, so as a topological space $|U| \simeq |\Spec{B_0}| \simeq |\Spf{B}|$. Let $a \in B$ and $U_a \subset U$ be the non-vanishing locus of $a$, or equivalently the non-vanishing locus of $\pi(a) \in B_0$ in $U = \Spec{B_0}$. Observe that $j\inv(U_a) = V_{a}$ (we write just $a$ instead of more formally correct $\ttilde{g}(a)$ for clarity). Then $\Oh_{\XX_Y}(U_a)$ fits into the exact sequence:
    \begin{center}
        \begin{tikzcd}
            0 \arrow[]{r}{} &
            \Oh_{\XX_Y}(U_a) \arrow[]{r}{\ttilde{g} \oplus \ttilde{\pi}} &
            A\langle a\inv \rangle \oplus B_0[\pi(a)\inv] \arrow[]{r}{\pi - g} &
            A_0[\pi(a)\inv] \arrow[]{r}{} &
            0
        \end{tikzcd}
    \end{center}
    Let us now look at the short exact sequence (\ref{eq:ses1}). We can tensor it with $B[a\inv]$ over $B$ (which is flat) obtaining:
    \begin{center}
        \begin{tikzcd}
            0 \arrow[]{r}{} &
            B [a\inv]  \arrow[]{r}{\ttilde{g} \oplus \ttilde{\pi}} &
            A[ a\inv ] \oplus B_0[\pi(a)\inv] \arrow[]{r}{\pi - g} &
            A_0[\pi(a)\inv] \arrow[]{r}{} &
            0
        \end{tikzcd}
    \end{center}
    Observe that $\varpi A \subset B$ so also $\varpi A[a\inv] \subset B[a\inv]$. Hence $B[a\inv]$ is open in $A[a\inv]$. Thus the previous exact sequence stays exact after completion by \cite[Proposition 0.7.4.8]{FK}. So we get:
    \begin{center}
        \begin{tikzcd}
            0 \arrow[]{r}{} &
            B \langle a\inv \rangle \arrow[]{r}{\ttilde{g} \oplus \ttilde{\pi}} &
            A\langle a\inv \rangle \oplus B_0[\pi(a)\inv] \arrow[]{r}{\pi - g} &
            A_0[\pi(a)\inv] \arrow[]{r}{} &
            0
        \end{tikzcd}
    \end{center}
    Hence $B\tate{a\inv} \simeq \Oh_{\XX_Y}(U_a)$, which was the desired relation.

    \textit{(ii)} Let $\ttilde{h} \in B$ be such that $\pi' (\ttilde{h}) = h$. We will now show that $A \simeq B[\ttilde{h}\inv] \simeq B\langle\ttilde{h}\inv \rangle$. If $a \in A \backslash B$ then $\pi(a) \in A_0 \backslash B_0$, so there exists aa positive integer $m$, such that $h^m \cdot \pi(a) \in B_0$. Then $\pi(\ttilde{h}^m \cdot a) \in B_0$, so $\ttilde{h}^m \cdot a \in B$, since $B = \pi\inv(B_0)$. So $A \simeq B[\ttilde{h}\inv]$, but since $A$ is $\varpi$-adically complete, then also $A \simeq (B[\ttilde{h}\inv])^{\wedge} \simeq B\langle \ttilde{h}\inv \rangle$.

    \textit{(iii)}
        Assume we are in the situation, where we are given a solid diagram as below:
    \begin{center}
        \begin{tikzcd}
            & & \ZZ \\
            \XX \arrow[swap]{r}{} \arrow[bend left=25]{urr}{f}&
            \XX_Y \arrow[dashed, swap]{ur}{\exists! \ttilde{h}} \\
            \XX\urs \arrow[]{u}{} \arrow[]{r}{j} &
            Y \arrow[]{u}{\iota} \arrow[bend right=25, swap]{uur}{h}
        \end{tikzcd}
    \end{center}
    Note that vertical arrows are homeomorphisms. So as a map of topological spaces, $\ttilde{h}$ is uniquely determined, as it is simply equal to $h$. Thus it is enough to show unique $\ttilde{h}^{\#}$ for the map of sheaves of topological rings. We have the following solid diagram:
    \begin{center}
        \begin{tikzcd}
            & & \Oh_{\ZZ} \arrow[bend right=25]{dll}{} \arrow[bend left=25]{ddl}{} \arrow[dashed]{dl}{\exists! \ttilde{h}^{\#}}\\
            f_* \Oh_{\XX} \arrow[]{d}{}&
            h_* \Oh_{\XX_Y} \arrow[]{d}{} \arrow[]{l}{} \\
            f_* \Oh_{\XX\urs}&
            h_* \Oh_{Y} \arrow[]{l}{}
        \end{tikzcd}
    \end{center}
    Pushforwards commute with limits, so the left-hand square on the diagram is a pullback square. Hence, by the universal property of the pullback, there exists a unique map of sheaves $\ttilde{h}^{\#}$.
\end{proof}

\begin{rmk}\begin{enumerate}[(i)]
    \item The conclusion \textit{(i)} in the proposition above is true under weaker assumptions. We only need to know that $\XX\urs \to Y$ is an affine schematically dominant morphism of schemes.
    \item Instead of taking the reduced special fiber we could have taken any closed subscheme of $V(I)$ given by killing some nilpotents, where $I$ is the ideal of definition of $X$, and the proof will not change. In particular, it applies to our compactified formal schemes. 
\end{enumerate}
\end{rmk}

\bibliography{references}
\end{document}